\documentclass{amsart}
\DeclareFontFamily{OML}{script}{}
\DeclareFontShape{OML}{script}{m}{it}
{ <5-20> rsfs10 }{}
\DeclareMathAlphabet{\mathscript}{OML}{script}{m}{it}

\renewcommand{\mathcal}[1]{{\mathscript #1}\hspace{0.2ex}}
\usepackage{color}
\ifx\red\undefined
\newcommand{\red}{\color{red}}
\fi
\usepackage{graphicx,graphics}
\usepackage{cite}
\usepackage{pifont}
\usepackage{amsthm}

\usepackage{amscd}
\usepackage{amsmath}
\usepackage{latexsym}
\usepackage{amsfonts}
\usepackage{amssymb}
\usepackage{color}
\usepackage{multicol}

\allowdisplaybreaks[4]

\newcommand{\re}[1]{\mbox{\rm$($\ref{#1}$)$}}

\newcommand{\m}{\hspace{1em}}
\newcommand{\mm}{\hspace{2em}}
\newcommand{\p}{\partial}

\def\s{\sigma}
\renewcommand{\epsilon}{\varepsilon}
\renewcommand{\t}{\widetilde}

\ifx\text\undefined
\newcommand{\text}{\mbox}
\fi
\ifx\operatorname\undefined
\newcommand{\operatorname}{\mathop}
\fi

\newcommand\be{\begin{equation}}
\newcommand\ee{\end{equation}}
\newcommand\bea{\begin{eqnarray}}
\newcommand\eea{\end{eqnarray}}
\newcommand\beaa{\begin{eqnarray*}}
\newcommand\eeaa{\end{eqnarray*}}
\newcommand{\dif}{\mathrm{d}}
\newcommand*{\dv}{\mathrm{div}}

\newenvironment{eqa}{\begin{equation}%
  \begin{array}{rcl}}{\end{array}\end{equation}}

\newcommand\beqa{\begin{eqa}}
\newcommand\eeqa{\end{eqa}}

\numberwithin{equation}{section}

\renewcommand{\tilde}{\widetilde}
\renewcommand{\hat}{\widehat}
\renewcommand{\bar}{\overline}

\newtheorem{thm}{Theorem}[section]

\newtheorem{lem}{Lemma}[section]

\newtheorem{rem}{Remark}[section]

\newcommand{\void}[1]{}

\numberwithin{equation}{section}

\begin{document}
\title{The impact of time delay in a tumor model\footnote{\today}
}
\author{Xinyue Evelyn Zhao}\author{Bei Hu}
\address{Department of Applied Computational Mathematics and Statistics,
University of Notre Dame, Notre Dame, IN 46556, USA}
\email{xzhao6@nd.edu, b1hu@nd.edu}
\maketitle

\begin{abstract}
In this paper we consider a free boundary tumor growth model with a time delay in cell proliferation and study how time delay affects the stability and the size of the tumor. The model is a coupled system of an elliptic equation, a parabolic equation and an ordinary differential equation. It incorporates the cell location under the presence of time delay, with the tumor boundary as a free boundary. A parameter $\mu$ in the model is proportional to the ``aggressiveness'' of the tumor. It is proved that there exists a unique classical radially symmetric stationary solution $(\sigma_*, p_*, R_*)$ which is stable for any $\mu > 0$ with respect to all radially symmetric perturbations (c.f. \cite{delay1}). However, under non-radially symmetric perturbations, we prove that there exists a critical value $\mu_*$, such that if $\mu<\mu_*$ then the stationary solution $(\sigma_*, p_*, R_*)$ is linearly stable; whereas if $\mu>\mu_*$ the stationary solution is unstable. It is actually unrealistic to expect the problem to be stable for large tumor aggressiveness parameter, therefore our result is more reasonable. Furthermore, we established that adding the time delay in the model would result in a larger stationary tumor, and if the tumor aggressiveness parameter is larger, then the time delay would have a greater impact on the size of the tumor.
\end{abstract}

\section{Introduction}
Over the last few decades, an increasing number of PDE models describing solid tumor growth in forms of free boundary problems have been proposed and studied. All these models provided a better and deeper understanding of the tumor growth.  The basic reaction-diffusion tumor model was studied in Greenspan \cite{G1, G2}, Cui and Escher \cite{CE1, CE3}, Escher and Matioc \cite{Escher4}, Bazaliy and Friedman \cite{BVF, BVF1}, Friedman and Hu\cite{FH3, FH4}, Friedman and Reitich \cite{FR1, FR2}. Furthermore, the basic model can be extended to more sophisticated ones by adding different factors. For example, Byrne and Chaplain \cite{BC},  Cui \cite{Cui1}, Cui and Friedman \cite{CF2}, Wang \cite{Zejia}, Wu and Zhou \cite{WZ1, WZ2}, Xu {\em et.\ al.}  \cite{delay1, delay5, delay6} analyzed the tumor growth under the effect of inhibitor; Friedman, Hu and Kao \cite{cellcycle1, cellcycle2, cellcycle3, cellcycle4} considered a multiscale tumor model by adding cell cycle; and Friedman {\em et.\ al.} \cite{FFH1, angio2, angio1} added the effect of angiogenesis. See also \cite{A1, C4, F4, FH5, RCS, other2, other3, Fengjie, Hongjing, W} for other extensions to a variety of different tumor models.

One biological meaningful extension of the basic tumor model is to add the effect of {\bf time delay $\tau$}. In real life, time delays can arise everywhere, since every process, whether it is long or short, would consume time. Time delays can represent gestation times, incubation periods, transport delays, or can simply lump complicated biological processes together, accounting for the time required for these processes to complete. The basic tumor model can be viewed as an approximation of model with time delay, since time delay $\tau$ is rather small compared with the time range $[0,T]$ we consider. However, compared with the basic model, model with time delay are more accurate and consistent with real life.

Here we propose a tumor growth model with time delay in cell proliferation. The time delay is reflected between the time at which a cell commences mitosis and the time at which the daughter cells are produced (it takes approximately 24 hours). In this model, oxygen and glucose are viewed as nutrients, with its concentration $\sigma$ satisfying the reaction-diffusion equation
\begin{equation}\label{1.1}
\lambda\sigma_t =\Delta\sigma   -\sigma\quad\text{in the tumor region }\Omega(t),
\end{equation}
where $-\sigma$ is the nutrients consumed by the tumor. Since the diffusion rate of oxygen or glucose (e.g. $\sim\text{1 }\text{min}^{-1}$) is much faster than the rate of cell proliferation (e.g. $\sim\text{1 }\text{day}^{-1}$), $\lambda$ is very small and can sometimes be set to be 0 (quasi-steady state approximation).

By conservation of mass, cell proliferation rate $S = \dv\vec{V}$, where $\vec{V}$ denotes the velocity field of cell movement within the tumor. Due to the presence of time delay, the tumor grows at a rate which is related to the nutrient concentration when it starts mitosis. For a simple approximation, we assume a linear relationship between the cell proliferation rate and the nutrient concentration:
\begin{equation}\label{1.2}
S = \mu[\sigma(\xi(t-\tau;x,t),t-\tau)-\tilde{\sigma}],
\end{equation}
(it is also a first order Taylor expansion for fully nonlinear model), where $\tilde{\sigma}$ is a threshold concentration, $\mu$ is a parameter expressing the ``intensity" of tumor expansion, and $\xi(s;x,t)$ represents the cell location at time $s$ as cells are moving with the velocity field $\vec{V}$. The function $\xi(s;x,t)$ satisfies the ODE
\begin{equation}\label{1.3}
\left \{
\begin{split}
&\frac{\dif \xi}{\dif s} = \vec{V}(\xi,s),\quad t-\tau\le s\le t,\\
&\xi\big{|}_{s=t} = x.
\end{split}
\right.
\end{equation}
In another word, $\xi$ tracks the path of the cell currently located at $x$. In this problem, (\ref{1.3}) describes how cells, including both interior cells and cells on the boundary, move due to the presence of time delay. Adding time delay to the basic tumor growth model makes our problem more reasonable and yet more challenging.

Furthermore, if the tumor is assumed to be of porous medium type where Darcy's law (i.e., $\vec{V}=-\nabla p$, where $p$ is the pressure, here we consider extracellular matrix as ``porous medium" in which cell moves) can be used, then
\begin{equation}\label{1.4}
-\Delta p = \mu[\sigma(\xi(t-\tau;x,t),t-\tau)-\tilde{\sigma}].
\end{equation}
Assuming the velocity field is continuous up to the boundary, we obtain the normal velocity of the moving boundary, namely,
\begin{equation}\label{1.5}
V_n = -\nabla p\cdot n = -\frac{\partial p}{\partial n}\quad\text{on }\partial \Omega(t).
\end{equation}
In addition, assume $\sigma$ and $p$ satisfy the boundary conditions:
\begin{eqnarray}\label{1.6}
&& \sigma = 1 \quad \text{ on } \partial \Omega(t),\\
\label{1.7}
&& p = \kappa \quad \text{ on } \partial \Omega(t),
\end{eqnarray}
where $\kappa$ is the mean curvature. Equation (\ref{1.6}) represents nutrient supply at the boundary and equation (\ref{1.7}) represents cell-to-cell adhesiveness.

Finally, it remains to prescribe initial conditions. Instead of defining the initial conditions at time 0, for this time-delay problem,
we are required to supply  the initial conditions on an interval $[-\tau, 0]$. For simplicity
we assume initial data are time independent on the interval $[-\tau,0]$:
\begin{eqnarray}\label{1.8}
&& \Omega(t)  = \Omega_0 \quad -\tau\le t \le 0, \\
\label{1.9}
&& \sigma(x,t) = \sigma_0(x) \quad\text{ in }\Omega_0,\quad -\tau\le t \le 0.
\end{eqnarray}

Now the problem is reduced to mainly finding two unknown functions $\sigma$ and $p$, together with the unknown tumor region $\Omega(t)$:
\begin{eqnarray}
&&\lambda\sigma_t - \Delta\sigma + \sigma = 0,  \hspace{2em} x\in\Omega(t), \, t>0,\label{1.10}\\
&&-\Delta p = \mu[\sigma(\xi(t-\tau;x,t),t-\tau)-\tilde{\sigma}],\hspace{2em} x\in\Omega(t), \, t>0,\label{1.11}\\
&&\label{1.12}\left \{
\begin{array}{lr}
\displaystyle
\frac{\dif \xi}{\dif s} = -\nabla p(\xi,s), \hspace{2em} t-\tau\le s\le t,\\
\xi = x, \hspace{2em}\hspace{2em}\hspace{2em}\hspace{2em} s = t,\\
\end{array}
\right.
\\
&&\sigma = 1, \hspace{2em} x\in\partial\Omega(t), \, t>0,\label{1.13}\\
&&p=\kappa,\hspace{2em} x\in\partial\Omega(t), \, t>0,\label{1.14}\\
&&V_n = -\frac{\partial p}{\partial n}, \hspace{2em} x\in\partial\Omega(t), \, t>0,\label{1.15}\\
&&\Omega(t) = \Omega_0, \hspace{2em} -\tau\le t \le 0,\label{1.16}\\
&&\sigma(x,t) = \sigma_0(x), \hspace{2em} x\in\Omega_0, \, -\tau\le t \le 0.\label{1.17}
\end{eqnarray}

We shall reformulate the radially symmetric case in the next section.
The radially symmetric model with time delay was studied
in \cite{delay2, delay4, delay3, delay1, delay5}, it will be justified in the next section that the models in
these papers are first-order approximations of our model in radially symmetric case. In \cite{delay1}, Xu, Zhou, and Bai rigorously proved that the stationary solution is always stable with respect to all radially symmetric perturbations in the case $\lambda = 0$. In reality, however, we cannot ensure that perturbation is strictly radially symmetric, thus it is more natural and reasonable to ask the stability under non-radial perturbations. In this paper, we shall consider the linear stability of the unique radially symmetric stationary solution $(\sigma_*, p_*, R_*)$ with respect to non-radial perturbations in the quasi-steady state case $\lambda = 0$; the existence of such a solution is guaranteed:
\begin{thm}\label{thm1.1}
The system \re{1.10}--\re{1.15} admits a unique radially symmetric classical stationary solution $(\sigma_*, p_*, R_*)$.
\end{thm}
Next assume the initial conditions are perturbed as follows:
\begin{eqnarray}
\partial \Omega(t): r = R_* + \epsilon\rho_0(\theta), && \sigma(r,\theta,t) = \sigma_*(r) + \epsilon w_0(r,\theta),\hspace{2em} -\tau\le t\le0.
\end{eqnarray}
Substituting
\begin{eqnarray}
     & &\partial \Omega(t): r = R_* + \epsilon \rho(\theta,t)+O(\epsilon^2),\nonumber\\
    &&\sigma(r,\theta,t) = \sigma_*(r) + \epsilon w(r,\theta,t)+O(\epsilon^2), \\
    &&p(r,\theta,t)= p_*(r) + \epsilon q(r,\theta,t)+O(\epsilon^2).\nonumber
\end{eqnarray}
into \re{1.10}--\re{1.15} and collecting the $\epsilon$-order terms, we obtain a linearized system around the unique radial solution $(\sigma_*, p_*, R_*)$, and establish the following results:

\begin{thm}\label{thm1.2}
There exists a critical value $\mu_*>0$ such that for any $\mu < \mu_*$, the radially symmetric stationary solution $(\sigma_*, p_*, R_*)$ is linearly stable in the sense
\begin{eqa}\label{1.20}
|\rho(\theta,t) - (a_1 \cos (\theta) + b_1 \sin(\theta))| \le C e^{-\delta t}, \hspace{2em} t>0,
\end{eqa}
for some constants $a_1$, $b_1$ and $\delta >0$. If $\mu > \mu_*$, this stationary solution is linearly unstable.
\end{thm}

\begin{rem}
The system \re{1.10}--\re{1.17} is invariant under coordinate translations, that is the reason why we exclude $a_1 \cos (\theta) + b_1 \sin(\theta)$ in \re{1.20}. \void{And we shall show the stationary solution $(\sigma_*, p_*, R_*)$ is linearly stable under translation of the origin when $\mu<\mu_*$.}
\end{rem}

\begin{rem}
In this paper, we consider only 2 space dimensional case; the 3 space dimensional case can be considered in a similar manner without any difficulties, but the computations will be much more involved.
\end{rem}

\begin{rem}
In contrast to the result in \cite{delay1}, where stability holds for all $\mu$ with respect to radially symmetric perturbations, our $\mu_*$ is finite, and instability stems from $n=2$ mode when $\mu > \mu_*$. Recall that $\mu$ represents the tumor aggressiveness, and it is
a biologically reasonable result that larger tumor aggressiveness induces instability.
\end{rem}

Here are a couple of interesting results of the impact of time delay on the size and stability of the stationary tumor:
\begin{thm}\label{thm1.3}
Adding the time delay would result in a larger stationary tumor as compared to the system without
delay. The biger the tumor proliferation intensity $\mu$ is, the greater impact that time delay has on the size of the stationary tumor.
\end{thm}

\begin{thm}\label{thm1.4}
Adding the time delay to the system would not alter the critical value $\mu_*$ for which the stability of the stationary solution changes.
\end{thm}

The outline of this paper is as follows. In Section 2, we reformulate the radially symmetric case. In Section 3, we establish the existence and uniqueness of the radially symmetric stationary solution $(\sigma_*, p_*, R_*)$. In Section 4, we introduce the linearization of the system at $(\sigma_*, p_*, R_*)$ and carry out the details
of our lengthy proofs of Theorems \ref{thm1.2},  \ref{thm1.3} and \ref{thm1.4}.


\section{Radially Symmetric Case}
In radially symmetric case, the  system \re{1.10}--\re{1.17} becomes
\begin{eqnarray}
&&\lambda\sigma_t(r,t) - \Delta\sigma(r,t) + \sigma(r,t) = 0,\hspace{2em} \text{in }B_{R(t)}, \; t>0,\label{2.1}\\
&&-\Delta p(r,t) = \mu [\sigma(\xi(t-\tau;r,t),t-\tau)-\tilde{\sigma}],\hspace{2em} \text{in }B_{R(t)}, \; t>0,\label{2.2}\\
&&\label{2.3} \left \{
\begin{array}{lr}
\displaystyle
\frac{\dif \xi}{\dif s} = -\frac{\partial p}{\partial r}(\xi,s), \hspace{2em} t-\tau\le s\le t,\\
\xi = r, \hspace{2em}\hspace{2em}\hspace{2em}\hspace{2em} s = t,\\
\end{array}
\right.
\\
&&\frac{\partial \sigma}{\partial r}(0,t) = 0,\hspace{2em} \sigma(R(t),t) = 1,\label{2.4}\\
&&\frac{\partial p}{\partial r}(0,t) = 0,\hspace{2em} p(R(t),t) = \frac{1}{R(t)},\label{2.5}\\
&&\frac{\dif R}{\dif t}   = -\frac{\partial p}{\partial r}(R(t),t),\label{2.6}\\
&&R(t)=R_0,\hspace{2em} -\tau\le t\le 0,\label{2.7}\\
&&\sigma(r,t)=\sigma_0(r),\hspace{2em}0\le r\le R_0,\; -\tau\le t\le 0,\label{2.8}
\end{eqnarray}
where $B_{R(t)}$ denotes the disk centered at 0 with radius $R(t)$. By integrating \re{2.2} over $B_{R(t)}$ and using \re{2.6}, we obtain
(recall that the space dimension is $2$)
\begin{equation}\label{2.9}
\begin{split}
R'(t) &= \frac{\mu}{|\partial B_{R(t)}|}\int_{B_{R(t)}}\Big[\sigma(\xi(t-\tau;r,t),t-\tau)-\tilde{\sigma}\Big]\dif V\\
&= \frac{\mu}{R(t)}\Big [\int_0^{R(t)} \sigma(\xi(t-\tau;r,t),t-\tau) r \dif r - \int_0^{R(t)} \tilde{\sigma} r \dif r\Big ].
\end{split}
\end{equation}

We shall make a substitution $r'=\xi(t-\tau,r,t)$ in the above integration, and derive some properties of $\xi(s,r,t)$ that will be needed in the substitution.
Taking another derivative with respect to $r$ on both sides of (2.3), we have
\begin{equation*}
\left \{
\begin{split}
&\frac{\dif}{\dif s}\Big(\frac{\partial \xi}{\partial r}\Big) = -\frac{\partial^2 p}{\partial r^2}(\xi,s)\frac{\partial \xi}{\partial r},\quad t-\tau\le s\le t,\\
&\frac{\partial \xi}{\partial r}\Big |_{s=t} = 1,
\end{split}
\right.
\end{equation*}
from which we find that $r' = \xi(t-\tau;r,t)$ satisfies
$$\dif r = \dif r' \exp\Big \{\int_{t-\tau}^t - \frac{\partial^2 p}{\partial r^2}(\xi(s;r,t),s)\dif s\Big\} = (1+O(\tau))\dif r'.$$
Furthermore, the domain of integration $r\in(0,R(t))$ becomes $r'\in(0,R(t-\tau))$ after changing variable.
This is justified in three steps as follows.
\begin{itemize}
\item[1.] From $\frac{\partial \xi}{\partial r}=\exp\big\{\int_t^s -\frac{\partial^2 p}{\partial r^2}(\xi,c)\dif c\big\}>0$, we find that $r'=\xi(t-\tau;r,t)$ is a monotone increasing function of $r$.
\item[2.] If $r=0$, then $r'=\xi(t-\tau;0,t)=0$. As a matter of fact, $\xi \equiv 0$ is the unique solution to the ODE (here it is clear that $\frac{\partial p}{\partial r}(0,s)= 0$)
\begin{equation*}
\left \{
\begin{split}
&\frac{\partial \xi}{\partial s} = -\frac{\partial p}{\partial r}(\xi,s),\\
&\xi\big|_{s=t} = 0.
\end{split}
\right.
\end{equation*}
\item[3.] Now we claim $\xi(t-\tau;R(t),t)=R(t-\tau)$.
Indeed, both $\xi(s,R(t),t)$ and $R(s)$ satisfy the same ODE
\begin{equation*}
\left \{
\begin{split}
&\frac{\partial \xi}{\partial s} = -\frac{\partial p}{\partial r}(\xi,s),\\
&\xi\big|_{s=t} = R(t).
\end{split}
\right.
\end{equation*}
\void{
and the free boundary $R(t)$ satisfies
\begin{equation*}\tag{**}
\left \{
\begin{split}
&\frac{\dif R(s)}{\dif t} = -\frac{\partial p}{\partial r}(R(s),s),\\
&R\big|_{s=t} = R(t).
\end{split}
\right.
\end{equation*}
}
By the uniqueness of the ODE solution, we derive $\xi(s;R(t),t)=R(s)$. Letting $s=t-\tau$, we conclude $\xi(t-\tau;R(t),t)=R(t-\tau)$.
\end{itemize}
Thus we conclude that the domain of integration under the change of variable
$r'=\xi(t-\tau,r,t)$   becomes $(0,R(t-\tau))$.
Integrating \re{2.3} over the interval $[t-\tau,t]$, we get
$$r = r' + \int_{t-\tau}^t -\frac{\partial p(\xi(c;r,t),c)}{\partial r} \dif c = r' + O(\tau),$$
and substitute it into \re{2.9} to obtain
\begin{equation}\label{2.10}
\begin{split}
R'(t)&=\frac{\mu}{R(t)}\Big
 [\int_0^{R(t-\tau)} \sigma(r',t-\tau)(r'+O(\tau))(1+O(\tau))\dif r' - \int_0^{R(t)} \tilde{\sigma} r \dif r \Big ]\\
&=\frac{\mu}{R(t)}\Big [\int_0^{R(t-\tau)} \sigma(r,t-\tau)r\dif r - \int_0^{R(t)} \tilde{\sigma} r\dif r + O(\tau) \Big].
\end{split}
\end{equation}

 The time delay $\tau$ in our model is actually very small; it is therefore reasonable
 to drop the $O(\tau)$ terms from \re{2.10} in radially symmetric case, with this approximation  our model
  coincides with those in \cite{delay2, delay4, delay3, delay1, delay5}.

\section{Radially Symmetric Stationary Solution}
In this section we prove the existence and uniqueness of the radially symmetric stationary solution $(\sigma_*(r),p_*(r),$ $ R_*)$ to the system \re{1.10}--\re{1.15}. After setting all $t$-derivative terms to be $0$, a stationary solution $(\sigma_*(r),p_*(r),R_*)$ satisfies
\begin{eqnarray}
&&-\Delta \sigma_*(r) + \sigma_*(r) = 0, \hspace{2em} r< R_*,\label{3.1}\\
&&-\Delta p_*(r) = \mu[\sigma_*(\xi(-\tau;r,0))-\tilde{\sigma}], \hspace{2em} r< R_*,\label{3.2}\\
&&\label{3.3}\left \{
\begin{array}{lr}
\displaystyle
\frac{\dif \xi}{\dif s}(s;r,0) = -\frac{\partial p_*}{\partial r}(\xi(s;r,0)), \hspace{2em} -\tau\le s\le 0,\\
\xi(s;r,0) = r, \hspace{2em}\hspace{2em}\hspace{2em}\hspace{2em} s = 0,\\
\end{array}
\right.
\\
&&\sigma_* = 1,\; p_* = \frac{1}{R_*},\hspace{2em} r=R_*,\label{3.4}\\
&&\int_0^{R_*} \Big[\sigma_*(\xi(-\tau;r,0)) - \tilde{\sigma}\Big]r\dif r = 0.\label{3.5}
\end{eqnarray}

\begin{lem}
For sufficiently small $\tau$, there exists a unique classical solution $(\sigma_*, p_*, R_*)$ to the problem \re{3.1}-\re{3.5}.
\end{lem}

\begin{proof}
To begin with, we introduce a change of variables
\begin{equation*}
\hat{r} = \frac{r}{R_*},\hspace{2em} \hat{\sigma}(\hat{r})=\sigma(r),\hspace{2em} \hat{p}(\hat{r})=R_*p(r), \hspace{2em} \hat{\xi}(s;\hat{r},0) = \frac{\xi(s;r,0)}{R_*}.
\end{equation*}
Solving the ODE \re{3.3} and substituting in \re{3.2} and \re{3.5}, we obtain a new system in the fixed domain $\{\hat{r}<1\}$. After dropping the $\;\hat{ }\;$ in the above variables, the new PDE system takes the following form:
\begin{eqnarray}
&&\label{3.6}\left \{
\begin{array}{lr}
\displaystyle
\frac{\dif \xi}{\dif s}(s;r,0) = -\frac{1}{R_*^3}\frac{\partial p}{\partial r}(\xi(s;r,0)), \hspace{2em} -\tau\le s\le 0,\\
\xi(s;r,0) = r, \hspace{2em}\hspace{2em}\hspace{2em}\hspace{2em} s = 0,\\
\end{array}
\right.
\\
&&-\Delta_r \sigma + R_*^2 \sigma = 0, \hspace{1em} \sigma(1) = 1,\label{3.7}\\
&&-\Delta_r p = \mu R_*^3 \Big[\sigma\Big(r+\frac{1}{R_*^3}\int_{-\tau}^0 \frac{\partial p}{\partial r}(\xi(s;r,0))\dif s \Big) - \tilde{\sigma}\Big], \hspace{1em} p(1) = 1,\label{3.8}\\
&&\int_0^1 \Big[\sigma\Big(r+\frac{1}{R_*^3}\int_{-\tau}^0 \frac{\partial p}{\partial r}(\xi(s;r,0))\dif s\Big)-\tilde{\sigma}\Big]r\dif r = 0.\label{3.9}
\end{eqnarray}

We first solve \re{3.7} explicitly as
\begin{equation}\label{3.10}
\sigma(r;R_*) = \frac{I_0(R_* r)}{I_0(R_*)}.
\end{equation}

Take $R_{\min}$ and $R_{\max}$ to be determined later. For any $R_{\min} \le R_* \le R_{\max}$, it is clear that $\sigma$ is uniquely determined by \re{3.10}. Substituting \re{3.10} into \re{3.8}, we shall prove that $p$ is also uniquely determined when $R_*$ is bounded by using contraction mapping principle.

{\red Note that $\xi(s;r,0)$ ($-\tau\le s \le 0$) might be out of the region $[0,1]$ by following the ODE \re{3.6} even if $0\le \xi(0,s,0)=r<1$ (i.e., started within the unit disk); however $\xi(s;r,0)$ should be very close to $r$ if $\tau$ is small enough. It is natural to assume $\xi(s;r,0)$ locates within $[0,2]$, so we take
$$\mathcal{P} = \{ p\in W^{2,\infty}[0,2]; \|p\|_{W^{2,\infty}[0,2]} \le M\}.$$
}For each $p\in \mathcal{P}$, we first solve $\xi$ from \re{3.6}, and substitute it into \re{3.8}, hence we shall obtain a unique solution $\bar{p}\in W^{2,\infty}[0,1]$ from the following system:
\begin{equation}
    \label{3.11}
    -\Delta_r \bar{p} = \mu R_*^3 \Big[\sigma\Big(r+\frac{1}{R_*^3}\int_{-\tau}^0 \frac{\partial p}{\partial r}(\xi(s;r,0))\dif s; R_* \Big) - \tilde{\sigma}\Big], \hspace{1em} \bar{p}(1) = 1.
\end{equation}
It follows from integrating \re{3.11} that
\begin{equation}
    \label{estimate1}
    \Big\|\frac1r\;\frac{\partial \bar{p}}{\partial r}\Big\|_{L^\infty[0,1]} \le \frac{\mu}{2} (R_{\max})^3 (\sigma_{\max}+\tilde{\sigma}),
\mm \|\bar{p}\|_{L^\infty[0,1]} \le 1 + \frac{\mu}{4}(R_{\max})^3 (
\sigma_{\max}+\tilde{\sigma}),
\end{equation}
\begin{equation}
    \label{estimate2}
    \Big \| \frac{\partial^2 \bar{p}}{\partial r^2} \Big \|_{L^\infty[0,1]} \le \frac{3\mu}{2}(R_{\max})^3 (\sigma_{\max}+\tilde{\sigma}),
\end{equation}
{\red where $\sigma_{\max} = \max\limits_{0\le r \le 2}\sigma(r;R_*)$. Note that $\bar{p}$ derived from \re{3.11} is only defined for $r\le 1$, we shall extend $\bar{p}$ to a bigger region. Define an extension of $\bar{p}$ as
\begin{equation}
    \label{extend}
    \tilde{p}(r) =
    \left\{
    \begin{split}
        &\bar{p}(r),\hspace{2em} &r\le 1,\\
        &\bar{p}(1) + \bar{p}'(1)(r-1),\hspace{2em} &1<r\le 2.
    \end{split}
    \right.
\end{equation}
It is easy to check $\|\tilde{p}\| \in W^{2,\infty}[0,2]$, and $\|\tilde{p}\|_{W^{2,\infty}[0,2]} \le 2 \|\bar{p}\|_{W^{2,\infty}[0,1]}$. Combining with \re{estimate1} and \re{estimate2}, we have
\begin{equation}
\label{estimate3}
    \|\tilde{p}\|_{W^{2,\infty}[0,2]} \le 2 \max\Big\{\frac{3\mu}{2} (R_{\max})^3 (\sigma_{\max}+\tilde{\sigma}), 1 + \frac{\mu}{4}(R_{\max})^3 (\sigma_{\max}+\tilde{\sigma})\Big\} \triangleq M_1
\end{equation}
}Define the mapping $\mathcal{L}: p \rightarrow \tilde{p}$. If we choose $M \ge M_1$, then by \re{estimate3}, $\tilde{p}\in \mathcal{P}$. Thus $\mathcal{L}$ maps $\mathcal{P}$ to itself. In the next step, we shall prove that $\mathcal{L}$ is a contraction.

Let $p_1, p_2 \in \mathcal{P}$, we solve $\xi_1, \xi_2$ from the following two systems:
\begin{eqnarray}
&&\label{3.12}\left \{
\begin{array}{lr}
\displaystyle
\frac{\dif \xi_1}{\dif s}(s;r,0) = -\frac{1}{R_*^3}\frac{\partial p_1}{\partial r}(\xi_1(s;r,0)), \hspace{2em} -\tau\le s\le 0,\\
\xi_1(s;r,0) = r, \hspace{2em}\hspace{2em}\hspace{2em}\hspace{2em} s = 0,\\
\end{array}
\right.
\end{eqnarray}
\begin{eqnarray}
&&\label{3.13}\left \{
\begin{array}{lr}
\displaystyle
\frac{\dif \xi_2}{\dif s}(s;r,0) = -\frac{1}{R_*^3}\frac{\partial p_2}{\partial r}(\xi_2(s;r,0)), \hspace{2em} -\tau\le s\le 0,\\
\xi_2(s;r,0) = r, \hspace{2em}\hspace{2em}\hspace{2em}\hspace{2em} s = 0.\\
\end{array}
\right.
\end{eqnarray}
Integrating \re{3.12} and \re{3.13} in $s$ and making a subtraction, we obtain the following estimate
\begin{equation*}
    \begin{split}
        |\xi_1-\xi_2| &\le \tau \frac{1}{R_*^3} \Big[\Big|\frac{\partial p_1}{\partial r}(\xi_1)-\frac{\partial p_2}{\partial r}(\xi_1)\Big| + \Big|\frac{\partial p_2}{\partial r}(\xi_1)-\frac{\partial p_2}{\partial r}(\xi_2)\Big|\Big]\\
        &\le \tau \frac{1}{R_*^3}\Big[ \|p_1 - p_2\|_{W^{2,\infty}[0,2]} + \|p_2\|_{W^{2,\infty}[0,2]}
         \max_{\substack{-\tau\le s\le 0\\0\le r\le 1}}|\xi_1 - \xi_2|\Big]\\
        &\le \tau \frac{1}{R_*^3}\|p_1 - p_2\|_{W^{2,\infty}[0,2]} + \tau \frac{1}{R_*^3}M\max_{\substack{-\tau\le s\le 0\\0\le r\le 1}}|\xi_1 - \xi_2|
    \end{split}
\end{equation*}
for all $-\tau \le s \le 0$ and $0\le r \le  1$,
hence
\begin{equation}
    \label{3.14}
    \max_{\substack{-\tau\le s\le 0\\0\le r\le 1}}|\xi_1 - \xi_2| \le \frac{\tau}{R_*^3 - \tau M}\|p_1 - p_2\|_{W^{2,\infty}[0,2]}.
\end{equation}
We then substitute $\xi_1, \xi_2$ into \re{3.11} and solve for $\bar{p}_1$ and $\bar{p}_2$, respectively. From \re{3.11}, $\bar{p}_1 - \bar{p}_2$ satisfies $(\bar{p}_1 - \bar{p}_2)(1) = 0$ and
$$-\Delta_r (\bar{p}_1 - \bar{p}_2) = \mu R_*^3\Big[\sigma\Big(r+\frac{1}{R_*^3}\int_{-\tau}^0 \frac{\partial p_1}{\partial r}(\xi_1(s;r,0))\dif s; R_* \Big) - \sigma\Big(r+\frac{1}{R_*^3}\int_{-\tau}^0 \frac{\partial p_2}{\partial r}(\xi_2(s;r,0))\dif s; R_* \Big)\Big]. $$
Using \re{3.14}, we have
\begin{equation*}
    \begin{split}
\Big\|\frac1r\frac{\partial (\bar{p}_1-\bar{p}_2)}{\partial r}\Big\|_{L^\infty[0,1]}  \le& \; \frac{\mu}{2}R_*^3 \Big\|\sigma\Big(r+\frac{1}{R_*^3}\int_{-\tau}^0 \frac{\partial p_1}{\partial r}(\xi_1(s))\dif s \Big)
 - \sigma\Big(r+\frac{1}{R_*^3}\int_{-\tau}^0 \frac{\partial p_2}{\partial r}(\xi_2(s))\dif s \Big)\Big\|_{L^\infty[0,1]}\\
\le& \;\frac{\mu}{2}R_*^3 \Big[ \Big\|\frac{\partial \sigma}{\partial r}\Big\|_{L^\infty[0,2]} \frac{1}{R_*^3}\int_{-\tau}^0 \Big(\frac{\partial p_1}{\partial r}(\xi_1(s)) - \frac{\partial p_2}{\partial r}(\xi_2(s))\Big)\dif s \Big]\\
\le& \;\frac{\mu}{2}\Big\|\frac{\partial \sigma}{\partial r}\Big\|_{L^\infty[0,2]}\tau \Big[\|p_1-p_2\|_{W^{2,\infty}[0,2]} + \|p_2\|_{W^{2,\infty}[0,2]} \max_{\substack{-\tau\le s\le 0\\0\le r<1}}|\xi_1 - \xi_2|\Big]\\
\le& \;M_2 \tau \|p_1-p_2\|_{W^{2,\infty}[0,2]},
    \end{split}
\end{equation*}
and similarly,
\begin{eqnarray*}
||\bar{p}_1-\bar{p}_2||_{L^\infty[0,1]} \le  M_3 \tau \|p_1-p_2\|_{W^{2,\infty}[0,2]},
\end{eqnarray*}
\begin{eqnarray*}
\Big\|\frac{\partial^2 (\bar{p}_1-\bar{p}_2)}{\partial r^2}\Big\|_{L^\infty[0,1]} \le  M_4 \tau \|p_1-p_2\|_{W^{2,\infty}[0,2]},
\end{eqnarray*}
where $M_2 = \frac{\mu}{2}\|\frac{\partial \sigma}{\partial r}\|_{L^\infty[0,2]}\Big(1+\frac{M\tau}{(R_{\min})^3-M\tau}\Big)$, $M_3 = \frac{\mu}{4}\|\frac{\partial \sigma}{\partial r}\|_{L^\infty[0,2]}\Big(1+\frac{M\tau}{(R_{\min})^3-M\tau}\Big)$ and $M_4 =\frac{3\mu}{2}\|\frac{\partial \sigma}{\partial r}\|_{L^\infty[0,2]}\Big(1+\frac{M\tau}{(R_{\min})^3-M\tau}\Big)$ are independent of $r$. It is clear that $M_4 > M_2 > M_3$, thus
\begin{equation}
    \label{estinew}
    \|\bar{p}_1-\bar{p}_2\|_{W^{2,\infty}[0,1]} \le M_4 \tau \|p_1-p_2\|_{W^{2,\infty}[0,2]}.
\end{equation}
{\red $\bar{p}_1$ and $\bar{p}_2$ are extended in the same way as in \re{extend}, hence
\begin{equation}
    \label{extend1}
    (\tilde{p}_1 -\tilde{p}_2) (r) =
    \left\{
    \begin{split}
        &(\bar{p}_1 -\bar{p}_2)(r),\hspace{2em} &r\le 1,\\
        &(\bar{p}_1 - \bar{p}_2)(1) + (\bar{p}_1 - \bar{p}_2)'(1)(r-1),\hspace{2em} &1<r\le 2.
    \end{split}
    \right.
\end{equation}
It is easy to derive $\tilde{p}_1-\tilde{p}_2 \in W^{2,\infty}[0,2]$ and $\|\tilde{p}_1-\tilde{p}_2\|_{W^{2,\infty}[0,2]}\le 2\|\bar{p}_1-\bar{p}_2\|_{W^{2,\infty}[0,1]}$. Combining with \re{estinew}, we have
\begin{equation}
    \label{estilast}
    \|\mathcal{L}p_1 - \mathcal{L}p_2\|_{W^{2,\infty}[0,2]} =  \|\tilde{p}_1-\tilde{p}_2\|_{W^{2,\infty}[0,2]} \le 2 M_4 \tau \|p_1-p_2\|_{W^{2,\infty}[0,2]}.
\end{equation}
We thus obtain a contraction mapping $\mathcal{L}$ by taking $\tau$ small so that
$2 M_4 \tau < 1$.}

Now for any particular $R_* \in [R_{\min}, R_{\max}]$, $\sigma$ and $p$ are uniquely determined, it remains to show that there exists a unique solution $R_*$ satisfying \re{3.9}. Substituting \re{3.10} into \re{3.9}, we find that it is equivalent to show that there exists a unique solution to the equation:
$$\int_0^1 \Big[\frac{I_0(rR + \frac{1}{R^2}\int_{-\tau}^0 \frac{\partial p}{\partial r}(\xi(s;r,0)) \dif s)}{I_0(R)} - \tilde{\sigma}\Big]r \dif r = 0.$$
In order to prove the above statement, we set
$$F(R,\tau)= \int_0^1 \Big[\frac{I_0(rR + \frac{1}{R^2}\int_{-\tau}^0 \frac{\partial p}{\partial r}(\xi(s;r,0)) \dif s)}{I_0(R)} - \tilde{\sigma}\Big]r \dif r,$$
then $F(R, 0) =\int_0^1 \Big[\frac{I_0(rR)}{I_0(R)} - \tilde{\sigma}\Big]r \dif r = P_0(R) - \frac{\tilde{\sigma}}{2}$,
where $P_0(R) = \frac{I_1(R)}{RI_0(R)}$. From \cite{Erdelyi} (pg.61) and \cite{FH3} (the equations (2.19) (2.21) and (2.26)), it is known that $P_0(R)$ is decreasing in $R$ and $0<P_0(R)\le \frac{1}{2}$. Since $0<\tilde{\sigma}<1$, there exists a unique solution to the equation $F(R,0) = 0$, denoted by $R_S$. In addition, from the fact that $F(R,0)$ is monotone decreasing in $R$, we have
$$ F(\frac{1}{2} R_S, 0) > 0, \hspace{1em} F(\frac{3}{2} R_S, 0) < 0.$$
Next we take derivative of $F(R,\tau)$ with respect to $R$ and expand the partial derivative in $\tau$ to get
\begin{equation}
\frac{\p F(R,\tau)}{\p R} = \frac{\p F(R,0)}{\p R} + \frac{\p^2 F(R,0)}{\p R\p \tau}\tau + O(\tau^2).
\end{equation}
When $\tau$ is small enough, the signs of $\frac{\partial F(R,\tau)}{\partial R}$ and  $\frac{\partial F(R,0)}{\partial R}$ should be the same, thus $F(R,\tau)$ is also monotone decreasing in $R$. In addition, from the fact that $F(R,\tau)$ is continuous in $\tau$, we have
$$ F(\frac{1}{2} R_S, \tau) > 0, \hspace{1em} F(\frac{3}{2} R_S, \tau) < 0.$$
Hence there exists a unique solution $R_*$ satisfying $F(R_*,\tau)=0$, i.e., equation \re{3.9}, when $\tau$ is small enough; furthermore we have $R_{\min} = \frac12 R_S < R_* < \frac32 R_S = R_{\max}$. The proof is complete.
\end{proof}

{\color{blue} \begin{rem} Since $\partial p_*/\partial r = 0 $ on the boundary for the
radially symmetric {\em  stationary solution}, we have $\xi(s;R_*,0)\equiv R_*$  and
$\xi(s;r, 0)$ will stay within the unit disk if initially  $\xi(0;r,0)=r<R_*$.
From  \re{3.2} it is clear that $\frac{\p p}{\p r}$ is not identically $0$ for $0<r<R_*$, therefore $\xi(s;r, 0)$ will not be a constant for $-\tau\le s\le 0$.

For our stationary solution, the free boundary does not move in time. But the velocity field
inside the tumor domain is not zero, and  movements are necessary to
 replace dead cells with new daughter cells to reach an equilibrium. Because of the time delay,
such replacement requires a   time $\tau$ for the  mitosis to complete
 and for the daughter cells to move  into the right place;
and that is incorporated into
the equation \re{3.3}. In that sense, the delay-time derivative cannot be set to zero
even for our stationary solution and our  solution differs from the classical
definition of stationary solution where time derivatives are all zero.
\end{rem}
}

Throughout the paper, we denote the corresponding radially symmetric stationary solution by $(\sigma_*,p_*,R_*)$.

\section{Linear Stability}
In this section, we consider the linearized problem of system \re{1.10}--\re{1.15}
with $\lambda = 0$ and shall determine a critical value $\mu_*$ such that $(\sigma_*,p_*,R_*)$ is linearly stable in the interval $0<\mu<\mu_*$ and linearly unstable for $\mu>\mu_*$. We shall also discuss the impact of time delay $\tau$ on the stability and the size of tumor.

We begin by making some small non-radially symmetric perturbations on the initial conditions (Note that the perturbations are made in a time interval $[-\tau,0]$ instead of an initial time due to the presence of time delay, and we assume for simplicity that the perturbation is uniform on the interval $[-\tau,0])$:
\begin{eqnarray}
&&\partial \Omega(t): r = R_* + \epsilon\rho_0(\theta),\hspace{2em} -\tau\le t\le0,\label{4.1}\\
&&\sigma(r,\theta,t) = \sigma_*(r) + \epsilon w_0(r,\theta),\hspace{2em} -\tau\le t\le0\label{4.2}.
\end{eqnarray}
To linearize \re{1.10}--\re{1.15}, we let
\begin{eqnarray}
     & &\partial \Omega(t): r = R_* + \epsilon \rho(\theta,t)+O(\epsilon^2),\label{4.3}\nonumber\\
    &&\sigma(r,\theta,t) = \sigma_*(r) + \epsilon w(r,\theta,t)+O(\epsilon^2), \label{thm4.3}\\
    &&p(r,\theta,t)= p_*(r) + \epsilon q(r,\theta,t)+O(\epsilon^2).\label{4.4}\nonumber
\end{eqnarray}

Since we are considering a domain which is a small perturbation of a disk, we shall express $\xi(s;r,\theta,t)$ in equation \re{1.12} in  polar coordinates $(\xi_1(s;r,\theta,t),\xi_2(s;r,\theta,t))$, where $\xi_1$ represents radius, and $\xi_2$ represents angle. Thus, the vector $\xi$ is expressed in the form $\xi = \xi_1 \vec{e}_1(\xi)$, where $\vec{e}_1(\xi) = \cos(\xi_2)\vec{i} + \sin(\xi_2)\vec{j}$ and $\vec{e}_2(\xi) = -\sin(\xi_2)\vec{i} + \cos(\xi_2)\vec{j}$ are the two basis vectors in polar coordinates. We then expand  $\xi_1,\xi_2$ in $\epsilon$ as
\be\label{thm4.4} \left\{
\begin{array}{c}
\xi_1 = \xi_{10} + \epsilon \xi_{11} + O(\epsilon^2),\\
\xi_2 = \xi_{20} + \epsilon \xi_{21} + O(\epsilon^2).
\end{array} \right.
\ee
Accordingly, $-\nabla $ is also expressed in polar coordinates, i.e., $-\nabla = -\vec{e}_1\frac{\partial}{\partial r} - \frac{1}{r}\vec{e}_2\frac{\partial }{\partial \theta}.$ Since $\frac{\dif \vec{e}_1(\xi)}{\dif s} = \vec{e}_2(\xi) \frac{\dif \xi_2}{\dif s},$
equation \re{1.12} is equivalent to
\begin{equation*}
\begin{split}
\frac{\dif \xi}{\dif s} &= \frac{\dif (\xi_1 \vec{e}_1(\xi))}{\dif s} = \frac{\dif \xi_1}{\dif s}\vec{e}_1(\xi) + \xi_1 \frac{\dif \xi_2}{\dif s}\vec{e}_2(\xi) \\
&= -\nabla p = -\frac{\partial p}{\partial r}\vec{e}_1(\xi) - \frac{1}{\xi_1}\frac{\partial p}{\partial \theta}\vec{e}_2(\xi) ,
\end{split}
\end{equation*}
from which we obtain two sets of ODEs in polar coordinates:
\vspace{-20pt}
\begin{multicols}{2}
\begin{equation*}
\left \{
\begin{split}
&\frac{\dif \xi_1}{\dif s} = -\frac{\partial p}{\partial r}(\xi_1,\xi_2,s),\quad t-\tau\le s\le t,\\
&\xi_1\Big|_{s=t} = r;
\end{split}
\right.
\end{equation*}

\begin{equation*}
\left \{
\begin{split}
&\frac{\dif \xi_2}{\dif s} = -\frac{1}{(\xi_1)^2}\frac{\partial p}{\partial \theta}(\xi_1,\xi_2,s),\quad t-\tau\le s\le t,\\
&\xi_2\Big|_{s=t} = \theta.
\end{split}
\right.
\end{equation*}
\end{multicols}

\noindent Substituting \re{thm4.3} and \re{thm4.4} into the above ODEs,
and dropping higher order terms, we   get
\begin{equation}\label{4.5}
\left \{
\begin{split}
&\frac{\dif \xi_{10}}{\dif s} = -\frac{\partial p_*}{\partial r}(\xi_{10}),\hspace{2em} t-\tau\le s \le t,\\
&\xi_{10}\Big|_{s=t} = r;
\end{split}
\right.
\end{equation}

\begin{equation}\label{4.6}
\left \{
\begin{split}
&\frac{\dif \xi_{11}}{\dif s} = -\frac{\partial^2 p_*}{\partial r^2}(\xi_{10})\xi_{11} - \frac{\partial q}{\partial r}(\xi_{10},\xi_{20},s),\hspace{2em} t-\tau\le s \le t,\\
&\xi_{11}\Big|_{s=t} = 0;
\end{split}
\right.
\end{equation}

\begin{equation}\label{4.7}
\left \{
\begin{split}
&\frac{\dif \xi_{20}}{\dif s} = 0,\hspace{2em} t-\tau\le s \le t,\\
&\xi_{20}\Big|_{s=t} = \theta;
\end{split}
\right.
\end{equation}

\begin{equation}\label{4.8}
\left \{
\begin{split}
&\frac{\dif \xi_{21}}{\dif s} = -\;\frac1{(\xi_{10})^2} \frac{\partial q}{\partial \theta}(\xi_{10}, \xi_{20}, s), \hspace{2em} t-\tau\le s \le t\\
&\xi_{21}\Big |_{s=t} = 0.
\end{split}
\right.
\end{equation}
Note that the equation for $\xi_{10}$ is the same as the equation for $\xi$ in radially symmetric case (i.e., \re{4.5} and \re{2.3} are the same), thus $\xi_{10}$ is independent of $\theta$; and from \re{4.7} we can easily derive $\xi_{20} \equiv \theta$.

Substituting \re{4.3} and \re{4.5}--\re{4.8} into \re{1.10}--\re{1.15}, using also the mean-curvature
formula in the 2-dimensional case for the curve $r=\rho$:
\[
 \kappa = \frac{\rho^2+2\rho_\theta^2-\rho\cdot\rho_{\theta\theta}}{\big(\rho^2+(\rho_\theta)^2\big)^{3/2}},
\]
and collecting only the linear terms in $\epsilon$, we obtain the linearized system in $B_{R_*}$ ($B_{R_*}$ denotes the disk centered at 0 with radius $R_*$), namely,
\begin{eqnarray}
&&  \Delta w(r,\theta,t) = w(r,\theta,t),\label{4.9}\\
&&w(R_*,\theta,t)=-\frac{\partial \sigma_*}{\partial r}\Big|_{r=R_*}\rho(\theta,t),\label{4.10}\\
&&\Delta q(r,\theta,t) = -\mu\frac{\partial \sigma_*}{\partial r}(\xi_{10}(t-\tau;r,t))\xi_{11}(t-\tau;r,\theta,t)-\mu w(\xi_{10}(t-\tau;r,t),\theta,t-\tau),\label{4.11}\\
&&q(R_*,\theta,t)=-\frac{1}{R_*^2}(\rho(\theta,t)+\rho_{\theta\theta}(\theta,t)) \void{- \frac{\partial p_*}{\partial r}\Big|_{r=R_*}\rho(\theta,t)},\label{4.12}\\
&&\frac{\dif \rho}{\dif t} = -\frac{\partial^2 p_*}{\partial r^2}\Big|_{r=R_*}\rho(\theta,t) - \frac{\partial q}{\partial r}\Big|_{r=R_*},\label{4.13}
\end{eqnarray}
where the equations for $\xi_{10}$ and $\xi_{11}$ are listed in \re{4.5} and \re{4.6}, respectively. Since $\xi_{21}$ does not appear explicitly in \re{4.9}--\re{4.13}, it is not needed.

In what follows, we  seek  solutions of the form
\begin{eqnarray*}
&& w(r,\theta,t) = w_n(r,t)\cos(n\theta),\\
&& q(r,\theta,t) = q_n(r,t)\cos(n\theta),\\
&& \rho(\theta,t) = \rho_n(t)\cos(n\theta),\\
&& \xi_{11}(s;r,\theta,t) = \varphi_n(s;r,t)\cos(n\theta).
\end{eqnarray*}
Noting that in a similar manner, we can also seek  solutions of the form
\begin{eqnarray*}
&& w(r,\theta,t) = w_n(r,t)\sin(n\theta),\\
&& q(r,\theta,t) = q_n(r,t)\sin(n\theta),\\
&& \rho(\theta,t) = \rho_n(t)\sin(n\theta),\\
&& \xi_{11}(s;r,\theta,t) = \varphi_n(s;r,t)\sin(n\theta).
\end{eqnarray*}

Using the relation $\Delta = \partial_{rr} +\frac{1}{r}\partial_r + \frac{1}{r^2}\partial_{\theta\theta}$
in \re{4.9}--\re{4.13}, we obtain the following system in $B_{R_*}$:
\begin{gather}
 - \frac{\partial^2w_n(r,t)}{\partial r^2} - \frac{1}{r}\frac{\partial w_n(r,t)}{\partial r}+\Big(\frac{n^2}{r^2}+1\Big)w_n(r,t) = 0,\label{4.14}\\
w_n(R_*,t) = -\frac{\partial \sigma_*}{\partial r}\Big |_{r=R_*} \rho_n(t),\label{4.15}
\end{gather}
\begin{equation}\label{4.16}
    \begin{split}
        -\frac{\partial^2 q_n(r,t)}{\partial r^2}-\frac{1}{r}\frac{\partial  q_n(r,t)}{\partial r} + \frac{n^2}{r^2}q_n(r,t) =& \mu w_n(\xi_{10}(t-\tau;r,t),t-\tau)\\
&+ \mu\frac{\partial \sigma_*}{\partial r}(\xi_{10}(t-\tau;r,t))\varphi_n(t-\tau;r,t) ,
    \end{split}
\end{equation}
\vspace{-8pt}
\begin{gather}
q_n(R_*,t) = \frac{n^2-1}{R_*^2}\rho_n(t)\void{-\frac{\partial p_*}{\partial r}\Big|_{r=R_*}\rho_n(t)},\label{4.17}\\
\frac{\dif \rho_n(t)}{\dif t} = -\frac{\partial^2 p_*}{\partial r^2}\Big |_{r=R_*} \rho_n(t) - \frac{\partial q_n}{\partial r}\Big |_{r=R_*},\label{4.18}
\end{gather}
where the steady state solution $(\sigma_*, p_*, R_*)$ satisfies \re{3.1}--\re{3.5}, $\xi_{10}$ satisfies \re{4.5} and $\varphi_n$ satisfies the following equation:
\begin{equation}\label{4.19}
\left \{
\begin{split}
&\frac{\partial \varphi_n(s;r,t)}{\partial s} = -\frac{\partial^2 p_*}{\partial r^2}(\xi_{10})\varphi_n(s;r,t) - \frac{\partial q_n(\xi_{10},s)}{\partial r},\hspace{2em} t-\tau\le s \le t,\\
&\varphi_n\Big|_{s=t} = 0.
\end{split}
\right.
\end{equation}

\setcounter{subsection}{-1}
\subsection{Properties of Bessel Functions}
In the sequel, we shall use modified Bessel functions $I_n(\xi)$ for $n\ge 0$. For convenience, we collect some properties of these functions here.

Recall that the modified Bessel function $I_n(\xi)$ satisfies the differential equations
\begin{equation}\label{4.20}
    I''_n(\xi) + \frac{1}{\xi}I'_n(\xi) - \Big(1+\frac{n^2}{\xi^2}\Big) I_n(\xi) = 0,
\end{equation}
and is given by
\begin{equation}
    \label{4.21}
    I_n(\xi) = \Big(\frac{\xi}{2}\Big)^n \sum_{k=0}^\infty \frac{1}{k!\Gamma(n+k+1)}\Big(\frac{\xi}{2}\Big)^{2k}.
\end{equation}
By \cite{FFBessel, FH3, FR2}, $I_n(\xi)$ satisfies
\begin{gather}
I'_n(\xi) + \frac{n}{\xi}I_n(\xi) = I_{n-1}(\xi), \hspace{2em} n\ge 1, \label{4.22}\\
I'_n(\xi) - \frac{n}{\xi}I_n(\xi) = I_{n+1}(\xi), \hspace{2em} n\ge 0, \label{4.23}\\
\xi^{n+1} I_n(\xi) = \frac{\dif}{\dif \xi}(\xi^{n+1} I_{n+1}(\xi)), \hspace{2em} n\ge 0, \label{4.24}\\
I_{n-1}(\xi) - I_{n+1}(\xi) = \frac{2n}{\xi}I_n(\xi), \hspace{2em} n\ge 1, \label{4.25}\\
I_{n-1}(\xi)I_{n+1}(\xi) < I_n^2(\xi), \hspace{2em} \xi>0, \label{4.26}\\
I_{n-1}(\xi)I_{n+1}(\xi) > I_n^2(\xi) -\frac{2}{\xi}I_n(\xi)I_{n+1}(\xi), \hspace{2em} \xi>0, \label{4.27}\\
I_m(\xi) I_n(\xi) = \sum_{k=0}^{\infty}\frac{\Gamma(m+n+2k+1)(\xi/2)^{m+n+2k}}{k!\Gamma(m+k+1)\Gamma(n+k+1)\Gamma(m+n+k+1)},\label{4.28}
\end{gather}
These properties of $I_n(\xi)$ are needed in the subsequent discussions.

\subsection{Expansion in $\tau$.} It is impossible to solve the system \re{3.1}--\re{3.5}, \re{4.14}--\re{4.19} explicitly. However, we would like to study the impact of $\tau$ on this system.
Since the time delay $\tau$ is actually very small, we look for the expansion in $\tau$ for the system \re{3.1}--\re{3.5}, \re{4.14}--\re{4.19}. Let us denote
\begin{eqnarray*}
R_* &=& R_*^0 + \tau R_*^1 + O(\tau^2),\\
\sigma_* &=& \sigma_*^0 + \tau\sigma_*^1 + O(\tau^2),\\
p_* &=& p_*^0 +  \tau p_*^1  + O(\tau^2),\\
w_n &=& w_n^0 + \tau w_n^1 + O(\tau^2),\\
q_n &=& q_n^0 + \tau q_n^1 + O(\tau^2),\\
\rho_n &=& \rho_n^0 + \tau\rho_n^1  + O(\tau^2).
\end{eqnarray*}

In order to compute the expansion in $\tau$ for \re{4.14}--\re{4.19}, we need to compute $\frac{\p\sigma_*}{\p r}$
in \re{4.15} \re{4.16}, and $\frac{\p^2 p_*}{\p r^2}$ in \re{4.18} \re{4.19}. To do that, we expand system \re{3.1} --- \re{3.5} in $\tau$. It follows from \re{3.1} and \re{3.4} that
$$\sigma_*(r) = \frac{I_0(r)}{I_0(R_*)}= \frac{I_0(r)}{I_0(R_*^0)} +\tau \frac{I_0(r)(-I_1(R_*^0))R_*^1}{I_0^2(R_*^0)} + O(\tau^2),$$
and therefore,
\begin{eqnarray}
\sigma_*^0(r) &=& \frac{I_0(r)}{I_0(R_*^0)},\label{4.29}\\
\sigma_*^1(r) &=& -\frac{I_0(r)I_1(R_*^0)}{I^2_0(R_*^0)}R_*^1.\label{4.30}
\end{eqnarray}

To find $\frac{\p^2 p_*}{\p r^2}$, we start with \re{3.2} and \re{3.3}.  We first integrate equation \re{3.3} over the interval $(-\tau,0)$ to obtain
$$r - \xi(-\tau;r,0) = \int_{-\tau}^0 -\frac{\partial p_*}{\partial r}(\xi(s;r,0))\dif s,$$
i.e.,
$$\xi(-\tau;r,0) = r + \int_{-\tau}^0 \frac{\partial p_*}{\partial r}(\xi(s;r,0))\dif s =
 r + \tau \frac{\partial p_*^0(r)}{\partial r} + O(\tau^2).$$
We then substitute the above expression for $\xi(-\tau;r,0)$ into \re{3.2}, since
\begin{equation}\label{4.31}
    \begin{split}
    \sigma_*(\xi(-\tau;r,0)) &= \sigma_*\Big(r + \tau \frac{\partial p_*^0(r)}{\partial r} + O(\tau^2)\Big)\\
         &= \sigma_*^0\Big ( r + \tau \frac{\partial p_*^0(r)}{\partial r} \Big ) + \tau\sigma_*^1\Big ( r + \tau \frac{\partial p_*^0(r)}{\partial r} \Big ) + O(\tau^2)\\
        &= \sigma_*^0(r) + \tau\Big(\frac{\partial \sigma_*^0}{\partial r}(r)\frac{\partial p_*^0}{\partial r}(r) +  \sigma_*^1(r)\Big) + O(\tau^2),
    \end{split}
\end{equation}
we derive the equations for $p_*^0$ and $p_*^1$,
\begin{gather}
-\frac{\partial^2 p_*^0}{\partial r^2} -\frac{1}{r}\frac{\partial p_*^0}{\partial r} = \mu[\sigma_*^0 - \tilde{\sigma}],\label{4.32}\\
-\frac{\partial^2 p_*^1}{\partial r^2} - \frac{1}{r}\frac{\partial p_*^1}{\partial r} = \mu\frac{\partial \sigma_*^0}{\partial r}\frac{\partial p_*^0}{\partial r}+\mu\sigma_*^1.\label{4.33}
\end{gather}
The boundary condition $p_*(R_*) = \frac{1}{R_*}$ is expanded as follows:
\begin{equation*}
    p_*^0(R_*^0) + \tau\frac{\partial p_*^0}{\partial r}(R_*^0)R_*^1 + \tau p_*^1(R_*^0) + O(\tau^2) = \frac{1}{R_*^0} -\tau\frac{R_*^1}{(R_*^0)^2} + O(\tau^2).
\end{equation*}
Thus, we have
\begin{gather}
    p_*^0(R_*^0) = \frac{1}{R_*^0},\label{4.34}\\
    p_*^1(R_*^0) = -\frac{R_*^1}{(R_*^0)^2}-\frac{\partial p_*^0}{\partial r}(R_*^0)R_*^1.\label{4.35}
\end{gather}

Next we   expand the integral equation \re{3.5} using \re{4.31}:
\begin{equation}\label{4.36}
    \begin{split}
      0 & =  \int_0^{R_*}[\sigma_*(\xi(-\tau;r,0))-\tilde{\sigma}]    
      r\dif r\\
        &= \int_0^{R_*}[\sigma_*^0(r)-\tilde{\sigma}]r\dif r + \tau \int_0^{R_*^0}\Big[\frac{\partial \sigma_*^0}{\partial r}(r)\frac{\partial p_*^0}{\partial r}(r)+\sigma_*^1(r)\Big]r\dif r + O(\tau^2).
    \end{split}
\end{equation}
By \re{4.29}, the first part of \re{4.36} is integrated explicitly as
\begin{equation}\label{4.37}
    \begin{split}
        \int_0^{R_*}& [\sigma_*^0(r)-\tilde{\sigma}]r\dif r =  \int_0^{R_*}\Big[\frac{I_0(r)}{I_0(R_*^0)}-\tilde{\sigma}\Big]r\dif r
        = \frac{R_*I_1(R_*)}{I_0(R_*^0)}-\frac{\tilde{\sigma}}{2}(R_*)^2\\
        =& \frac{R_*^0 I_1(R_*^0)}{I_0(R_*^0)} - \frac{\tilde{\sigma}}{2}(R_*^0)^2 + \tau\Big[ \frac{R_*^1 I_1(R_*^0)}{I_0(R_*^0)}
        + \frac{R_*^0(I_0(R_*^0) + I_2(R_*^0))}{2I_0(R_*^0)}R_*^1-\tilde{\sigma}R_*^0 R_*^1\Big] + O(\tau^2).
    \end{split}
\end{equation}
Combining \re{4.36} and \re{4.37}, we derive
\begin{equation*}
    \begin{split}
      R_*^0 \Big( \frac{ I_1(R_*^0)}{I_0(R_*^0)}-\frac{\tilde{\sigma}}{2}R_*^0\Big) &+ \tau \Big [\frac{R_*^1 I_1(R_*^0)}{I_0(R_*^0)} +  \frac{R_*^0(I_0(R_*^0)+I_2(R_*^0))}{2I_0(R_*^0)}R_*^1 \\
        &- \tilde{\sigma}R_*^0 R_*^1 +  \int_0^{R_*^0}\Big(\frac{\partial \sigma_*^0}{\partial r}(r)\frac{\partial p_*^0}{\partial r}(r)+\sigma_*^1(r)\Big)r\dif r\Big] = O(\tau^2),
    \end{split}
\end{equation*}
which leads to a set of two equations,
\begin{gather}
    \frac{ I_1(R_*^0)}{I_0(R_*^0)}-\frac{\tilde{\sigma}}{2} R_*^0  = 0,\label{4.38}\\
    \frac{R_*^1 I_1(R_*^0)}{I_0(R_*^0)} +  \frac{R_*^0(I_0(R_*^0)+I_2(R_*^0))}{2I_0(R_*^0)}R_*^1 -\tilde{\sigma}R_*^0 R_*^1 + \int_0^{R_*^0}\Big(\frac{\partial \sigma_*^0}{\partial r}(r)\frac{\partial p_*^0}{\partial r}(r)+\sigma_*^1(r)\Big)r\dif r = 0.\label{4.39}
\end{gather}
These two equations determine $R_*^0$ and $R_*^1$, respectively.

Similarly, $w_n^0$ and $w_n^1$ satisfy the same equation \re{4.14}. Expanding \re{4.15} we find
\[
w_n^0(R_*^0+\tau R_*^1,t)+\tau w_n^1(R_*^0,t) = - \Big(\frac{\p\sigma_*^0}{\p r}(R_*^0+\tau R_*^1) +
\tau \frac{\p\sigma_*^1}{\p r}(R_*^0) \Big) [\rho_n^0(t) + \tau \rho_n^1(t) ] + O(\tau^2),
\]
which gives
\bea
 && w_n^0(R_*^0,t) = -\frac{\partial \sigma_*^0}{\partial r}(R_*^0)\rho_n^0(t), \label{4.40} \\
 && w_n^1(R_*^0,t) = -\frac{\partial w_n^0}{\partial r}(R_*^0,t)R_*^1-\frac{\partial \sigma_*^0}{\partial r}(R_*^0)\rho_n^1(t) - \frac{\partial^2 \sigma_*^0}{\partial r^2}(R_*^0)R_*^1 \rho_n^0(t) -\frac{\partial \sigma_*^1}{\partial r}(R_*^0)\rho_n^0(t). \label{4.41}
\eea

The next step is to expand \re{4.16} and \re{4.19} in $\tau$. Noting that
\begin{equation}\label{4.42}\begin{split}
\varphi_n(t-\tau;r,t) =& \varphi_n(t;r,t) + \frac{\partial \varphi_n}{\partial s}(t;r,t)(-\tau) + O(\tau^2)\\
=& 0 + \Big(-\frac{\partial^2 p_*}{\partial r^2}(\xi_{10})\varphi_n(t;r,t)-\frac{\partial q_n}{\partial r}(r,t)\Big)(-\tau) + O(\tau^2)\\
=& 0 + \Big(0 - \frac{\partial q_n^0}{\partial r}(r,t)\Big)(-\tau) + O(\tau^2)\\
=& \tau\frac{\partial q_n^0}{\partial r}(r,t) + O(\tau^2),
\end{split}
\end{equation}
and using \re{4.5},
\begin{equation}\label{4.43}
    \begin{split}
        \frac{\partial \sigma_*}{\partial r}(\xi_{10}(t-\tau;r,t))\varphi_n(t-\tau;r,t)
        &= \Big(\frac{\partial \sigma_*^0}{\partial r}(r) + O(\tau)\Big)\Big(\tau \frac{\partial q_n^0}{\partial r}(r,t) + O(\tau^2)\Big)\\
        &= \tau \frac{\partial \sigma_*^0}{\partial r}(r)\frac{\partial q_n^0}{\partial r}(r,t) + O(\tau^2),
    \end{split}
\end{equation}
we deduce,
\begin{equation}\label{4.44}
\begin{split}
    w_n(\xi_{10}(t-\tau;r,t),t-\tau) =& w_n^0(\xi_{10}(t-\tau;r,t),t-\tau) + \tau w_n^1(r,t) + O(\tau^2)\\
    =& w_n^0\Big(r+\int_{t-\tau}^t \frac{\partial p_*}{\partial r}(\xi_{10}(s;r,t))\dif s, t-\tau\Big) + \tau w_n^1(r,t) + O(\tau^2)\\
    =& w_n^0(r,t) + \tau\Big[\frac{\partial w_n^0}{\partial r}(r,t)\frac{\partial p_*^0}{\partial r}(r) - \frac{\partial w_n^0}{\partial t}(r,t) + w_n^1(r,t) \Big] + O(\tau^2).
    \end{split}
\end{equation}
Applying \re{4.42}--\re{4.44} into \re{4.16}, we derive equations for $q_n^0$ and $q_n^1$, respectively,
\begin{gather}
    -\frac{\partial^2 q_n^0}{\partial r^2}-\frac{1}{r}\frac{\partial q_n^0}{\partial r}+ \frac{n^2}{r^2}q_n^0 = \mu w_n^0,\label{4.45}\\
    -\frac{\partial^2 q_n^1}{\partial r^2}-\frac{1}{r}\frac{\partial q_n^1}{\partial r}+ \frac{n^2}{r^2}q_n^1 = \mu \frac{\partial \sigma_*^0}{\partial r}\frac{\partial q_n^0}{\partial r}+ \mu\frac{\partial w_n^0}{\partial r}\frac{\partial p_*^0}{\partial r} -\mu\frac{\partial w_n^0}{\partial t} + \mu w_n^1.\label{4.46}
\end{gather}
To get the boundary condition for $q_n^0$ and $q_n^1$, we write \re{4.17} as
\[
q_n^0(R_*^0+\tau R_*^1,t)+\tau q_n^1(R_*^0,t) = \frac{n^2-1}{(R_*^0+\tau R_*^1)^2}[\rho_n^0(t) + \tau \rho_n^1(t) ] + O(\tau^2),
\]
hence
\bea
 &&  q_n^0(R_*^0,t) = \frac{n^2-1}{(R_*^0)^2}\rho_n^0(t),  \label{4.47}\\
 && q_n^1(R_*^0,t)=-\frac{\partial q_n^0}{\partial r}(R_*^0,t)R_*^1 + \frac{n^2-1}{(R_*^0)^2}\rho_n^1(t) - \frac{2(n^2-1)R_*^1}{(R_*^0)^3}\rho_n^0(t). \label{4.48}
\eea

Finally from \re{4.18} we have
\[
 \frac{\dif}{\dif t}[\rho_0^0(t)+\tau \rho_n^1(t)] =
 - \Big(\frac{\p^2 p_*^0}{\p r^2}(R_*^0+\tau R_*^1) +
\tau \frac{\p^2 p_*^1}{\p r^2}(R_*^0) \Big) [\rho_n^0(t) + \tau \rho_n^1(t) ]
 - \frac{\p( q_n^0+\tau q_n^1)}{\p r}(R_*^0+\tau R_*^1) + O(\tau^2),
\]
which implies
\bea\label{4.49}
&& \frac{\dif \rho_n^0(t)}{\dif t}  = -\frac{\partial ^2 p_*^0}{\partial r^2}(R_*^0)\rho_n^0(t) - \frac{\partial q_n^0}{\partial r}(R_*^0,t), \\
&&     \begin{split}\label{4.50}
    \frac{\dif \rho_n^1(t)}{\dif t}= &-\frac{\partial^2 p_*^0}{\partial r^2}(R_*^0)\rho_n^1(t)-\frac{\partial^3 p_*^0}{\partial r^3}(R_*^0)R_*^1 \rho_n^0(t)\\
    &-\frac{\partial^2 p_*^1}{\partial r^2}(R_*^0)\rho_n^0(t)-\frac{\partial^2 q_n^0}{\partial r^2}(R_*^0,t)R_*^1 - \frac{\partial q_n^1}{\partial r}(R_*^0,t).
    \end{split}
\eea

We now group all the zeroth-order terms and the first-order terms in $\tau$, respectively, leading to two separate systems.

\subsection{zeroth-order terms in $\tau$} Collecting the zeroth-order terms from \re{4.29}, \re{4.32}, \re{4.34}, \re{4.38}, \re{4.40}, \re{4.45}, \re{4.47}, and \re{4.49}, we obtain the following system in $B_{R_*^0}$,
\begin{gather}
-\frac{\partial^2 \sigma_*^0}{\partial r^2} - \frac{1}{r}\frac{\partial\sigma_*^0}{\partial r}  = - \sigma_*^0, \m \sigma_*^0(R^0_*)=1, \mm\mbox{i.e., } \m \sigma_*^0(r) = \frac{I_0(r)}{I_0(R_*^0)},\label{4.51}\\
\frac{ I_1(R_*^0)}{I_0(R_*^0)}-\frac{\tilde{\sigma}}{2} R_*^0  = 0, \mm\mbox{i.e., } \m \frac{\tilde{\sigma}}{2} = \frac{I_1(R_*^0)}{R_*^0 I_0(R_*^0)},\label{4.52}\\
-\frac{\partial^2 p_*^0}{\partial r^2} - \frac{1}{r}\frac{\partial p_*^0}{\partial r} = \mu[\sigma_*^0 - \tilde{\sigma}], \hspace{2em} p_*^0(R_*^0)=\frac{1}{R_*^0},\label{4.53}\\
-\frac{\partial^2 w_n^0}{\partial r^2}-\frac{1}{r}\frac{\partial w_n^0}{\partial r}+\Big (\frac{n^2}{r^2}+1\Big) w_n^0 = 0, \hspace{2em} w_n^0(R_*^0,t) = -\frac{\partial \sigma_*^0}{\partial r}(R_*^0)\rho_n^0(t),\label{4.54}\\
-\frac{\partial^2 q_n^0}{\partial r^2}-\frac{1}{r}\frac{\partial q_n^0}{\partial r}+ \frac{n^2}{r^2} q_n^0 = \mu w_n^0, \hspace{2em} q_n^0(R_*^0,t) = \frac{n^2-1}{(R_*^0)^2}\rho_n^0(t),\label{4.55}\\
\frac{\dif \rho_n^0(t)}{\dif t} = -\frac{\partial ^2 p_*^0}{\partial r^2}(R_*^0)\rho_n^0(t) - \frac{\partial q_n^0}{\partial r}(R_*^0,t).\label{4.56}
\end{gather}

We first solve $p_*^0(r)$ and $w_n^0(r,t)$ from \re{4.53} and \re{4.54} as
\begin{gather}
    p_*^0(r) = \frac{1}{4}\mu \tilde{\sigma} r^2 - \mu\frac{I_0(r)}{I_0(R_*^0)} + \frac{1}{R_*^0} + \mu - \frac{1}{4}\mu\tilde{\sigma}(R_*^0)^2,\label{4.57}\\
    w_n^0(r,t) = -\frac{I_1(R_*^0)I_n(r)}{I_0(R_*^0)I_n(R_*^0)}\rho_n^0(t)\label{4.58},
\end{gather}
from which we compute the following terms needed in the subsequent computation,
\bea
    \label{4.59}
  && \frac{\partial w_n^0}{\partial r}(r,t) = -\frac{I_1(R_*^0)}{I_0(R_*^0)I_n(R_*^0)}\Big(I_{n+1}(r)+\frac{n}{r}I_n(r)\Big)\rho_n^0(t),\\
    \label{4.60}
  &&  \frac{\partial^2 p_*^0}{\partial r^2}(R_*^0) = \frac{1}{2}\mu\tilde{\sigma}-\mu\Big( 1- \frac{I_1(R_*^0)}{R_*^0 I_0(R_*^0)}\Big) = \mu\Big[\frac{2I_1(R_*^0)}{R_*^0 I_0(R_*^0)}-1\Big], \\
    \label{4.61}
  &&  \frac{\partial^3 p_*^0}{\partial r^3}(R_*^0) = \mu\Big[ \frac{1}{R_*^0} - \frac{2I_1(R_*^0)}{(R_*^0)^2 I_0(R_*^0)} - \frac{I_1(R_*^0)}{I_0(R_*^0)}\Big],
\eea
in deriving \re{4.60} we also made use of \re{4.52}.
To find $q_n^0$, let $\eta_n^0 = q_n^0 + \mu w_n^0$. Combining \re{4.54} and \re{4.55}, we find that $\eta_n^0$   satisfies
$$-\frac{\partial^2 \eta_n^0}{\partial r^2} - \frac{1}{r}\frac{\partial \eta_n^0}{\partial r} + \frac{n^2}{r^2}\eta_n^0 = 0,\quad \text{in }B_{R_*^0},$$
and its solution is given in the form
$$\eta_n^0(r,t) = C_1(t) r^n,$$
thus,
\begin{equation}\label{4.62}
    q_n^0(r,t) = \eta_n^0(r,t) - \mu w_n^0(r,t) = C_1(t)r^n - \mu w_n^0(r,t),
\end{equation}
where $C_1(t)$ is determined by the boundary condition \re{4.55}. Using also \re{4.58}, we get
\begin{equation}\label{4.63}
    C_1(t) = \frac{1}{(R_*^0)^n}\Big [ \frac{n^2-1}{(R_*^0)^2} - \mu \frac{I_1(R_*^0)}{I_0(R_*^0)}\Big]\rho_n^0(t).
\end{equation}
In order to calculate $\frac{\partial q_n^0}{\partial t}(R_*^0,t)$ in \re{4.56}, we use \re{4.23}, \re{4.58}, \re{4.62}, and \re{4.63} to obtain
\begin{equation}\label{4.64}
    \begin{split}
        \frac{\partial q_n^0}{\partial r}(R_*^0,t) &= C_1(t)n (R_*^0)^{n-1} - \mu\frac{\partial w_n^0}{\partial r}\Big |_{r=R_*^0}
        = \Big[ \frac{n(n^2-1)}{(R_*^0)^3} + \mu\frac{I_1(R_*^0)I_{n+1}(R_*^0)}{I_0(R_*^0)I_n(R_*^0)}\Big]\rho_n^0(t).
    \end{split}
\end{equation}
Taking another derivative with respect to $r$, we have
\begin{equation}
\begin{split}
    \label{4.65}
    \frac{\partial^2 q_n^0}{\partial r^2}(R_*^0,t)=&\; \Big[\frac{n(n-1)}{(R_*^0)^2}\Big(\frac{n^2-1}{(R_*^0)^2}-\frac{\mu I_1(R_*^0)}{I_0(R_*^0)}\Big) \\
    &\; - \frac{\mu I_1(R_*^0) I_{n+1}(R_*^0)}{R_*^0 I_0(R_*^0)I_n(R_*^0)} + \frac{\mu((R_*^0)^2 +n^2-n)I_1(R_*^0)}{(R_*^0)^2 I_0(R_*^0)}\Big]\rho_n^0(t),\hspace{2em} n\ge 2,
    \end{split}
\end{equation}
which will be needed in the subsequent calculations.
Now substituting \re{4.60} and \re{4.64} into \re{4.56}, we derive
\begin{equation*}
\begin{split}
    \frac{\dif \rho_n^0(t)}{\dif t}
    =\Big [ \mu\Big( 1 - \frac{2 I_1(R_*^0)}{R_*^0I_0(R_*^0)} - \frac{I_1(R_*^0)I_{n+1}(R_*^0)}{I_0(R_*^0)I_n(R_*^0)}\Big) - \frac{n(n^2-1)}{(R_*^0)^3} \Big] \rho_n^0(t),
\end{split}
\end{equation*}
which integrates to
\begin{equation}\label{4.66}
    \rho_n^0(t) = \rho_n^0(0)\exp\Big\{ \Big [\mu\Big( 1 - \frac{2 I_1(R_*^0)}{R_*^0I_0(R_*^0)} - \frac{I_1(R_*^0)I_{n+1}(R_*^0)}{I_0(R_*^0)I_n(R_*^0)}\Big) - \frac{n(n^2-1)}{(R_*^0)^3}\Big ] t \Big\}.
\end{equation}
We shall discuss the long-time behavior of $\rho_n^0(t)$ based on \re{4.66}. As will be seen, the analysis is different for $n=0$, $n=1$, and $n\ge 2$.

\begin{lem}\label{thm4.1}
For $n=0$ and any $\mu>0$, there exists $\delta>0$ such that $|\rho_0^0(t)|\le |\rho_0^0(0)|e^{-\delta t}$, for all $t>0$.
\end{lem}
\begin{proof}
When $n=0$, $\frac{n(n^2-1)}{(R_*^0)^3} = 0$, \re{4.66} becomes
$$\rho_0^0(t) = \rho_0^0(0)\exp\Big\{ \Big[ 1 - \frac{2 I_1(R_*^0)}{R_*^0I_0(R_*^0)}  - \frac{I_1^2(R_*^0)}{I_0^2(R_*^0)} \Big] \mu t\Big\}.$$
It suffices to show
\begin{equation}\label{4.67}
    1-\frac{2I_1(x)}{xI_0(x)}-\frac{I_1^2(x)}{I_0^2(x)} < 0, \hspace{2em} \text{for } x>0.
\end{equation}
This inequality is equivalent to (3.22) in \cite{Huang}, which has been established already.
\end{proof}
\begin{rem}\label{thmthm4.1}
$n=0$ represents radially-symmetric perturbations. Indeed, in this case
$$r = R_* + \epsilon \rho_0(t) = R_*^0 + \epsilon \rho_0^0(t) + \tau(R_*^1 + \epsilon\rho_0^1(t))+O(\tau^2).$$
When $\tau$ is small, we do not expect the first-order to have a major contribution, and the above result
is just another indication that the stability discussed in  \cite{delay1} is valid for all $\mu$.
\end{rem}

\begin{lem}\label{thm4.2}
For $n=1$ and any $\mu>0$, we have $\rho_1^0(t) = \rho_1^0(0)$, for all $t>0$.
\end{lem}
\begin{proof}
When $n=1$, $\frac{n(n^2-1)}{(R_*^0)^3} = 0$, by \re{4.66}, we have
\begin{equation*}
        \rho_1^0(t) = \rho_1^0(0)\exp\Big\{ \Big[ 1 - \frac{2 I_1(R_*^0)}{R_*^0I_0(R_*^0)}  - \frac{I_2(R_*^0)}{I_0(R_*^0)} \Big] \mu t\Big\}
            = \rho_1^0(0),
\end{equation*}
since $I_0(x)-I_2(x) = \frac{2}{x}I_1(x)$ by \re{4.25}.
\end{proof}

\begin{lem}\label{thmthm4.3}
For $n\ge 2$,
\begin{equation}\label{4.68}
    1-\frac{2I_1(x)}{x I_0(x)} - \frac{I_1(x)I_{n+1}(x)}{I_0(x)I_n(x)} > 0,\hspace{2em} x>0.
\end{equation}
\end{lem}

The proof of this lemma can be found in \cite{Huang} (Lemma 3.3). For $n\ge 2$, we define $\mu_n^0$ to be the solution to
$$\mu_n^0 \Big( 1 - \frac{2 I_1(R_*^0)}{R_*^0I_0(R_*^0)} - \frac{I_1(R_*^0)I_{n+1}(R_*^0)}{I_0(R_*^0)I_n(R_*^0)}\Big) - \frac{n(n^2-1)}{(R_*^0)^3} = 0, $$
that is,
\begin{equation}\label{4.69}
    \mu^0_n = \frac{\frac{n(n^2-1)}{(R_*^0)^3}}{1-\frac{2I_1(R_*^0)}{R_*^0 I_0(R_*^0)} - \frac{I_1(R_*^0)I_{n+1}(R_*^0)}{I_0(R_*^0)I_n(R_*^0)}}.
\end{equation}
Lemma \ref{thmthm4.3} implies that  $\mu^0_n > 0$. We then have the following lemma.

\begin{lem}\label{thmthm4.4}
For $n\ge 2$, $\mu_n^0 < \mu_{n+1}^0$.
\end{lem}
\begin{proof}
By \re{4.69}, we only need to establish the inequality
\begin{equation*}
    \frac{\frac{n(n^2-1)}{x^3}}{1-\frac{2I_1(x)}{x I_0(x)} - \frac{I_1(x)I_{n+1}(x)}{I_0(x)I_n(x)}} < \frac{\frac{(n+1)[(n+1)^2-1]}{x^3}}{1-\frac{2I_1(x)}{x I_0(x)} - \frac{I_1(x)I_{n+2}(x)}{I_0(x)I_{n+1}(x)}}, \hspace{2em} x>0.
\end{equation*}
Using Lemma \ref{thmthm4.3}, it suffices to show
\begin{equation}\label{4.70}
    (n+2)x\frac{I_{n+1}(x)}{I_n(x)}-(n-1)x\frac{I_{n+2}(x)}{I_{n+1}(x)}-3x\frac{I_0(x)}{I_1(x)} + 6<0,\hspace{2em} x>0.
\end{equation}
The above inequality has been established in \cite{Huang}. The proof is complete.
\end{proof}

Since Lemmas \ref{thm4.1} and \ref{thm4.2} are valid for all $\mu$, we define $\mu_0^0 = \mu^0_1 = \infty$. And set
\begin{equation}\label{4.71}
    \mu_* = \min\{\mu_0^0, \mu_1^0, \mu_2^0, \mu_3^0, \cdots\}.
\end{equation}
Then by  Lemma \ref{thmthm4.4},
\begin{equation}\label{4.72}
    \mu_* = \mu^0_2.
\end{equation}
Combining Lemmas \ref{thmthm4.3},  \ref{thmthm4.4} and equation \re{4.72}, it is easy to derive the following result.

\begin{lem}\label{thm4.5}
For $n \ge 2$ and $\mu < \mu_*$, there exists $\delta >0$ such that
\begin{equation}\label{4.73}
    |\rho_n^0(t)| \le |\rho_n^0(0)|e^{-\delta n^3 t}, \text{ for all } t>0,
\end{equation}
where $\delta$ is independent of $n$.
\end{lem}

\begin{proof}
Since $\mu < \mu_*$, there exists $\delta_1 >0$ independent of $n$ such that
\begin{equation*}
    \mu\Big( 1 - \frac{2 I_1(R_*^0)}{R_*^0I_0(R_*^0)} - \frac{I_1(R_*^0)I_{n+1}(R_*^0)}{I_0(R_*^0)I_n(R_*^0)}\Big) - \frac{n(n^2-1)}{(R_*^0)^3} < -\delta_1 n^3
\end{equation*}
is valid for $n$ sufficiently large, i.e., $n>n_0$. On the other hand, for each $n \in [2, n_0]$, there exists a corresponding $\delta_n > 0$ such that
\begin{equation*}
    \mu\Big( 1 - \frac{2 I_1(R_*^0)}{R_*^0I_0(R_*^0)} - \frac{I_1(R_*^0)I_{n+1}(R_*^0)}{I_0(R_*^0)I_n(R_*^0)}\Big) - \frac{n(n^2-1)}{(R_*^0)^3} < -\delta_n n^3.
\end{equation*}
Choosing $0< \delta < \min\{ \delta_1, \delta_2, \cdots, \delta_{n_0} \}$, we thus have
\begin{equation*}
    \mu\Big( 1 - \frac{2 I_1(R_*^0)}{R_*^0I_0(R_*^0)} - \frac{I_1(R_*^0)I_{n+1}(R_*^0)}{I_0(R_*^0)I_n(R_*^0)}\Big) - \frac{n(n^2-1)}{(R_*^0)^3} < -\delta n^3
\end{equation*}
holds for all $n\ge 2$. Therefore, it follows from \re{4.66} that
\begin{align*}
    |\rho_n^0(t)| &=|\rho_n^0(0)|\exp\Big\{ \Big [\mu\Big( 1 - \frac{2 I_1(R_*^0)}{R_*^0I_0(R_*^0)} - \frac{I_1(R_*^0)I_{n+1}(R_*^0)}{I_0(R_*^0)I_n(R_*^0)}\Big) - \frac{n(n^2-1)}{(R_*^0)^3}\Big ] t \Big\} \\
    &\le |\rho_n^0(0)| e^{-\delta n^3 t},\quad \text{for all }t>0.\qedhere
\end{align*} \end{proof}

\begin{rem}\label{thmthm4.2}
Lemma \ref{thm4.5} indicates that when the time delay $\tau$ is small enough, and tumor proliferation intensity $\mu$ is smaller than a critical value (i.e., $\mu < \mu_*$), then the stationary solution $(\sigma_*,p_*,R_*)$ is linearly stable even under non-radially symmetric perturbations. However, in contrast to the result in \cite{delay1},
we showed that the system is unstable with respect to perturbation when
$ \mu > \mu_*$. As indicated earlier, the instability comes from $n=2$ mode, which does not
contradict the result in \cite{delay1}.
\end{rem}

\subsection{Sign of $R_*^1$}
In 4.1, we have derived the equation for $R_*^1$. Since $R_* = R_*^0 + \tau R_*^1 + O(\tau^2)$, we would like to know  how the time delay $\tau$ affects the size of the tumor $R_*$, thus we are interested in the sign of $R_*^1$.

\begin{thm}\label{thmthmthm4.1}
$R_*^1 > 0$, and $R_*^1$ is monotone increasing in $\mu$.
\end{thm}

\begin{proof}
Substituting \re{4.29}, \re{4.30}, and \re{4.57} into \re{4.39}, recalling also the equality $\frac{\tilde{\sigma}}{2}=\frac{I_1(R_*^0)}{R_*^0 I_0(R_*^0)}$ from \re{4.52}, we obtain,
\begin{equation}\label{4.74}
\begin{split}
    R_*^1\Big[  &\frac{R_*^0(I_0(R_*^0)+I_2(R_*^0))}{2I_0(R_*^0)} - \frac{ I_1(R_*^0)}{I_0(R_*^0)}\Big] \\&+ \int_0^{R_*^0} \Big[\frac{I_1(r)}{I_0(R_*^0)}\Big(\frac{\mu I_1(R_*^0)}{R_*^0 I_0(R_*^0)}r - \mu\frac{I_1(r)}{I_0(R_*^0)}\Big) - \frac{I_0(r)I_1(R_*^0)}{I_0^2(R_*^0)}R_*^1\Big] r\dif r = 0.
    \end{split}
\end{equation}
It follows from \re{4.24} that
$$ r I_0(r) = \frac{\dif}{\dif r}(r I_1(r)), \hspace{2em} r^2 I_1(r) =\frac{\dif }{\dif r}(r^2 I_2(r)),$$
and by applying \re{4.22} and \re{4.25} we have
$$\frac{\dif}{\dif r}\Big[ \frac{1}{2}r^2(I_1^2(r) - I_0(r)I_2(r))\Big] = r I_1^2(r).$$
Using the above equations, we shall write the integral in \re{4.74} explicitly as
\begin{equation}\label{4.75}
    \begin{split}
        &\int_0^{R_*^0} \Big[\frac{I_1(r)}{I_0(R_*^0)}\Big(\frac{\mu I_1(R_*^0)}{R_*^0 I_0(R_*^0)}r - \mu\frac{I_1(r)}{I_0(R_*^0)}\Big) - \frac{I_0(r)I_1(R_*^0)}{I_0^2(R_*^0)}R_*^1\Big] r\dif r\\
        =&\; \frac{\mu I_1(R_*^0)}{R_*^0 I_0^2(R_*^0)}
        \int_0^{R_*^0} r^2 I_1(r)\dif r - \frac{\mu}{I_0^2(R_*^0)}\int_0^{R_*^0} r I_1^2(r)\dif r - \frac{I_1(R_*^0)R_*^1}{I_0^2(R_*^0)}\int_0^{R_*^0} r I_0(r)  \dif r\\
        =&\; \frac{\mu R_*^0 I_1(R_*^0)I_2(R_*^0)}{I_0^2(R_*^0)} - \frac{\mu (R_*^0)^2 I_1^2(R_*^0)}{2 I_0^2(R_*^0)} + \frac{\mu (R_*^0)^2 I_2(R_*^0)}{2I_0(R_*^0)} - \frac{R_*^0 I_1^2(R_*^0)}{I_0^2(R_*^0)}R_*^1\\
        =&\; \frac{\mu R_*^0}{2I_0^2(R_*^0)}  \Big(2I_1(R_*^0)I_2(R_*^0)- R_*^0 I_1^2(R_*^0) + R_*^0 I_0(R_*^0)I_2(R_*^0)\Big)- \frac{R_*^0 I_1^2(R_*^0)}{I_0^2(R_*^0)}R_*^1.
    \end{split}
\end{equation}
It then follows from \re{4.74} and \re{4.75} that
\begin{equation}\label{4.76}
\begin{split}
    R_*^1\frac{R_*^0}{2 I_0^2(R_*^0)} \Big[& -\frac{2I_0(R_*^0)I_1(R_*^0)}{R_*^0} +  I_0^2(R_*^0)  + I_0(R_*^0)I_2(R_*^0) - 2I_1^2(R_*^0)  \Big]\\
    &= \mu \frac{R_*^0}{2I_0^2(R_*^0)} \Big[ - 2I_1(R_*^0)I_2(R_*^0)+ R_*^0 I_1^2(R_*^0)  - R_*^0 I_0(R_*^0)I_2(R_*^0)\Big].
    \end{split}
\end{equation}
Let
$$A(x) = -\frac{2I_1(x)I_0(x)}{x} + I_0^2(x) + I_2(x)I_0(x) - 2 I_1^2(x) = 2(-I_1^2(x)+ I_0(x) I_2(x)),$$
$$B(x) = -2 I_1(x) I_2(x) + x I_1^2(x) - x I_2(x)I_0(x).$$
To determine the sign of $R_*^1$, we need to determine the signs of $A$ and $B$.

By \re{4.26} and \re{4.27}, it follows
$$I_0(x)I_2(x) < I_1^2(x),\hspace{2em} x>0,$$
$$I_0(x)I_2(x) > I_1^2(x) - \frac{2}{x}I_1(x)I_2(x), \hspace{2em} x>0,$$
hence
\begin{equation}\label{4.77}
    A(x) <0 ,\hspace{2em} x>0,
\end{equation}
\begin{equation}
    \label{4.78} B(x) < 0, \hspace{2em} x>0.
\end{equation}
From \re{4.76}
\begin{equation} \label{4.79}
    R_*^1 = \mu\frac{B(R_*^0)}{A(R_*^0)},
\end{equation}
so we can directly derive $R_*^1 > 0$ by using \re{4.77} and \re{4.78}. Furthermore, it is easy to tell that $R_*^1$ is monotone increasing in $\mu$.\end{proof}

\begin{rem}
Since $R_*^1 >0$, adding the time delay to the system would result in a larger stationary tumor.
It is pretty reasonable because compared with models without time delay, there is more time for the tumor to grow in models with time delay. Theorem \ref{thmthmthm4.1} also indicates that the biger the tumor proliferation intensity $\mu$ is, the greater impact that time delay has on the size of the tumor.
\end{rem}

\subsection{first-order terms in $\tau$}
In what follows, we are going to tackle the system involving all the first-order terms in $\tau$, and we are more interested in the impact of time delay $\tau$ on our system.
We now collect first-order equations and their respective boundary conditions
from \re{4.30}, \re{4.33}, \re{4.35}, \re{4.39}, \re{4.41}, \re{4.46}, \re{4.48} and \re{4.50}:
\begin{gather}
   -\frac{\partial^2 \sigma_*^1}{\partial r^2} - \frac{1}{r}\frac{\partial\sigma_*^1}{\partial r}  = - \sigma_*^1, \m
   \sigma_*^1(R_*^0) = -\frac{ I_1(R_*^0)}{I_0(R_*^0)}R_*^1,\mm\mbox{i.e., } \m
   \sigma_*^1(r) = -\frac{I_0(r)I_1(R_*^0)}{I^2_0(R_*^0)}R_*^1,\label{4.80}\\
    \frac{R_*^1 I_1(R_*^0)}{I_0(R_*^0)} +  \frac{R_*^0(I_0(R_*^0)+I_2(R_*^0))}{2I_0(R_*^0)}R_*^1 -\tilde{\sigma}R_*^0 R_*^1 + \int_0^{R_*^0}\Big(\frac{\partial \sigma_*^0}{\partial r}(r)\frac{\partial p_*^0}{\partial r}(r)+\sigma_*^1(r)\Big)r\dif r = 0,\label{4.81}\\
    -\frac{\partial^2 p_*^1}{\partial r^2} - \frac{1}{r}\frac{\partial p_*^1}{\partial r} = \mu\frac{\partial \sigma_*^0}{\partial r}\frac{\partial p_*^0}{\partial r}+\mu\sigma_*^1,\hspace{2em} p_*^1(R_*^0)=-\frac{R_*^1}{(R_*^0)^2}-\frac{\partial p_*^0}{\partial r}(R_*^0)R_*^1,\label{4.82}\\
   \left\{ \begin{gathered}\label{4.83}
        -\frac{\partial^2 w_n^1}{\partial r^2}-\frac{1}{r}\frac{\partial w_n^1}{\partial r}+\Big (\frac{n^2}{r^2}+1\Big) w_n^1 = 0,\\
        w_n^1(R_*^0,t) = -\frac{\partial w_n^0}{\partial r}(R_*^0,t)R_*^1-\frac{\partial \sigma_*^0}{\partial r}(R_*^0)\rho_n^1(t) - \frac{\partial^2 \sigma_*^0}{\partial r^2}(R_*^0)R_*^1 \rho_n^0(t) -\frac{\partial \sigma_*^1}{\partial r}(R_*^0)\rho_n^0(t),
    \end{gathered} \right.
    \\
    \left\{\begin{gathered}\label{4.84}
    -\frac{\partial^2 q_n^1}{\partial r^2}-\frac{1}{r}\frac{\partial q_n^1}{\partial r}+ \frac{n^2}{r^2} q_n^1 = \mu\frac{\partial \sigma_*^0}{\partial r}\frac{\partial q_n^0}{\partial r} + \mu\frac{\partial w_n^0}{\partial r}\frac{\partial p_*^0}{\partial r}-\mu\frac{\partial w_n^0}{\partial t} + \mu w_n^1,\\
    q_n^1(R_*^0,t)=-\frac{\partial q_n^0}{\partial r}(R_*^0,t)R_*^1 + \frac{n^2-1}{(R_*^0)^2}\rho_n^1(t) - \frac{2(n^2-1)R_*^1}{(R_*^0)^3}\rho_n^0(t),
    \end{gathered} \right.
\end{gather}
\vspace{-12pt}
\begin{equation}
    \begin{split}\label{4.85}
    \frac{\dif \rho_n^1(t)}{\dif t}=&-\frac{\partial^2 p_*^0}{\partial r^2}(R_*^0)\rho_n^1(t)-\frac{\partial^3 p_*^0}{\partial r^3}(R_*^0)R_*^1 \rho_n^0(t)\\
    &-\frac{\partial^2 p_*^1}{\partial r^2}(R_*^0)\rho_n^0(t)-\frac{\partial^2 q_n^0}{\partial r^2}(R_*^0,t)R_*^1 - \frac{\partial q_n^1}{\partial r}(R_*^0,t).
    \end{split}
\end{equation}

In \re{4.85}, $p_*^0$ and $q_n^0$ are already computed, we only need to compute $\frac{\partial^2 p_*^1}{\partial r^2}(R_*^0)$ and $\frac{\partial q_n^1}{\partial r}(R_*^0, t)$. Integrating \re{4.82} over $(0,r)$ with $r \dif r$, we obtain
\[
 \frac{\p p^1_*}{\p r} (r) = -\; \frac{\mu}r \int_0^r  \Big(\frac{\partial \sigma_*^0}{\partial r}(y)
  \frac{\partial p_*^0}{\partial r}(y)+ \sigma_*^1(y) \Big) y \dif y.
\]
We then substitute the expressions of $\s_*^0$ from \re{4.51}, $p_*^0$ from \re{4.57}, and $\s_*^1$ from \re{4.80} into the
above equality to derive
\begin{equation}\label{4.86}
    \begin{split}
     \frac{\p p^1_*}{\p r} (r) =&
     -\; \frac{\mu}r \int_0^r  \Big[ \frac{\mu I_1(y)}{I_0(R_*^0)}\Big(\frac12 \t\s y -  \frac{I_1(y)}{I_0(R_*^0)}
     \Big)  - \frac{I_0(y) I_1(R^0_*)}{I_0^2(R_*^0)} R_*^1\Big] y dy.
    \end{split}
\end{equation}
Since, by \re{4.22} and \re{4.23},
$$\frac\dif{\dif r} \Big(r^2 I_0(r) -2rI_1(r)\Big) = r^2 I_1(r),$$
$$\frac\dif{\dif r} \Big( \frac{r^2(I_1^2(r)-I_0^2(r))}2 + r I_0(r)I_1(r) \Big) = r I_1^2(r),$$
$$\frac\dif{\dif r} \Big(r  I_1(r)\Big) = r I_0(r),$$
the integral in \re{4.86} evaluates to
\[
    \begin{split}
        \frac{\p p^1_*}{\p r} (r)  =& \frac{\mu^2 \t\s}{2I_0(R_*^0)}[2I_1(r)-rI_0(r)]
        +  \frac{\mu^2  }{ I_0^2(R_*^0)}\Big( \frac{r(I_1^2(r)-I_0^2(r))}2 +  I_0(r)I_1(r) \Big) + \frac{\mu I_1(R_*^0)}{I_0^2(R_*^0)}R_*^1 I_1(r).
    \end{split}
    \]
Using \re{4.22}, \re{4.23}, and \re{4.25}, taking another derivative and evaluating at $R_*^0$, also recalling the equality $\frac{\tilde{\sigma}}{2} = \frac{I_1(R_*^0)}{R_*^0 I_0(R_*^0)}$ from \re{4.52}, we derive
\begin{equation}\label{4.87}
    \begin{split}
        \frac{\partial^2 p_*^1}{\partial r^2}(R_*^0) =&\;
         \frac{\mu^2 I_1(R_*^0)}{R_*^0 I_0^2(R_*^0)}\Big[- R_*^0 I_1(R_*^0) + I_2(R_*^0)
         \Big]+
         \frac{\mu^2}{2I_0^2(R_*^0)}\Big[I_0^2(R_*^0)-\frac{2I_0(R_*^0)I_1(R_*^0)}{R_*^0}\\ &\; + I_1^2(R_*^0)\Big]
         + \frac{\mu R_*^1 I_1(R_*^0)}{I_0^2(R_*^0)}\Big[I_0(R_*^0) - \frac{I_1(R_*^0)}{R_*^0}\Big].
    \end{split}
\end{equation}
We completed computation of $\frac{\partial^2 p_*^1}{\partial r^2}(R_*^0)$. We now proceed a long and tedious journey to compute $\frac{\partial q_n^1}{\partial r}(R_*^0,t)$. From \re{4.83}, $w_n^1(r,t)$ can be solved in the form
\begin{equation}\label{4.88}
w_n^1(r,t) = C_3(t)I_n(r).
\end{equation}
Substituting it into the boundary condition in \re{4.83}, using also \re{4.22}, \re{4.23}, \re{4.51}, \re{4.59}, and \re{4.80}, we derive
\begin{equation*}
    C_3(t) I_n(R_*^0) = \Big[\frac{I_1(R_*^0)I_{n+1}(R_*^0)}{I_0(R_*^0)I_n(R_*^0)} + \frac{(n+1)I_1(R_*^0)}{R_*^0 I_0(R_*^0)} - 1+ \frac{I_1^2(R_*^0)}{I_0^2(R_*^0)}\Big] R_*^1 \rho_n^0(t) - \frac{I_1(R_*^0)}{I_0(R_*^0)}\rho_n^1(t).
\end{equation*}
Thus $C_3(t)$ is uniquely determined, and
\begin{equation}\label{4.89}
\begin{split}
    w_n^1(r,t) =& \; C_3(t)I_n(r) =  -\frac{I_1(R_*^0)I_n(r)}{I_0(R_*^0)I_n(R_*^0)}\rho_n^1(t) + \\
    &\Big[\frac{I_1(R_*^0)I_{n+1}(R_*^0)}{I_0(R_*^0)I_n(R_*^0)} + \frac{(n+1)I_1(R_*^0)}{R_*^0 I_0(R_*^0)} - 1+ \frac{I_1^2(R_*^0)}{I_0^2(R_*^0)}\Big]\frac{I_n(r)}{I_n(R_*^0)}R_*^1\rho_n^0(t).
    \end{split}
\end{equation}

As in the computation of $q_n^0$ and $w_n^0$, we let $\eta_n^1 = q_n^1 + \mu w_n^1$. Combining \re{4.83} and \re{4.84}, we find that $\eta_n^1$ satisfies
\begin{equation}\label{4.90}
    -\frac{\partial^2 \eta_n^1}{\partial r^2}-\frac{1}{r}\frac{\partial \eta_n^1}{\partial r}+\frac{n^2}{r^2}\eta_n^1 = \mu\frac{\partial \sigma_*^0}{\partial r}\frac{\partial q_n^0}{\partial r} + \mu\frac{\partial w_n^0}{\partial r}\frac{\partial p_*^0}{\partial r}-\mu\frac{\partial w_n^0}{\partial t},\quad \text{in }B_{R_*^0},
\end{equation}
with the boundary condition
\begin{equation}\label{4.91}
        \eta_n^1(R_*^0,t) = q_n^1(R_*^0,t) + \mu w_n^1(R_*^0,t).
\end{equation}

For simplicity, let us denote the differential operator by $L_n:= -\partial_{rr} -\frac{1}{r}\partial_r + \frac{n^2}{r^2}$, and rewrite the solution $\eta_n^1$ to \re{4.90} and \re{4.91} as $\eta_n^1 = u_n^{(1)} + u_n^{(2)} + u_n^{(3)} + u_n^{(4)}$, where $u_n^{(1)}, u_n^{(2)}, u_n^{(3)}$ and $u_n^{(4)}$ satisfy the following equations, respectively.
\begin{equation}\label{4.92}
    \left \{
    \begin{split}
        &L_n u_n^{(1)} = \mu \frac{\partial \sigma_*^0}{\partial r}\frac{\partial q_n^0}{\partial r},\quad \text{in }B_{R_*^0},\\
        &u_n^{(1)}(R_*^0,t) = 0;
    \end{split}
    \right.
\end{equation}
\begin{equation}\label{4.93}
    \left \{
    \begin{split}
        &L_n u_n^{(2)} = \mu \frac{\partial w_n^0}{\partial r}\frac{\partial p_*^0}{\partial r},\quad \text{in }B_{R_*^0},\\
        &u_n^{(2)}(R_*^0,t) = 0;
    \end{split}
    \right.
\end{equation}
\begin{equation}\label{4.94}
    \left \{
    \begin{split}
        &L_n u_n^{(3)} = -\mu  \frac{\partial w_n^0}{\partial t},\quad \text{in }B_{R_*^0},\\
        &u_n^{(3)}(R_*^0,t) = 0;
    \end{split}
    \right.
\end{equation}
\begin{equation}\label{4.95}
    \left \{
    \begin{split}
        &L_n u_n^{(4)} = 0,\quad \text{in }B_{R_*^0},\\
        &u_n^{(4)}(R_*^0,t) = q_n^1(R_*^0,t) + \mu w_n^1(R_*^0,t).
    \end{split}
    \right.
\end{equation}

We start by analyzing $u_n^{(1)}$. Substituting \re{4.51} and \re{4.62} into \re{4.92}, recalling also \re{4.23}, \re{4.59}, and \re{4.63}, we derive the explicit form of \re{4.92}, namely,
\begin{equation}\label{4.96}
    \begin{split}
        L_n u_n^{(1)} =&\mu \frac{\partial \sigma_*^0}{\partial r}\frac{\partial q_n^0}{\partial r} = \mu\frac{I_1(r)}{I_0(R_*^0)}\Big[C_1(t)n r^{n-1} + \frac{I_1(R_*^0)}{I_0(R_*^0) I_n(R_*^0)}\Big(I_{n+1}(r)+\frac{n}{r}I_n(r)\Big)\rho_n^0(t)\Big] \\
        =& \mu\frac{I_1(r)}{I_0(R_*^0)}\Big[\frac{n r^{n-1}}{(R_*^0)^n} \Big( \frac{n^2-1}{(R_*^0)^2} - \mu\frac{I_1(R_*^0)}{I_0(R_*^0)}\Big ) + \mu\frac{I_1(R_*^0)I_{n+1}(r)}{I_0(R_*^0) I_n(R_*^0)} + \mu\frac{n I_1(R_*^0)I_n(r)}{rI_0(R_*^0) I_n(R_*^0)}\Big] \rho_n^0(t).\\
    \end{split}
\end{equation}
Recalling the definition of Bessel function $I_n(r)$ in \re{4.21}, we have $\lim\limits_{r\rightarrow 0} \frac{I_1(r)}{r} = \frac12$ and  $\lim\limits_{r\rightarrow 0} \frac{I_n(r)}{r} = 0$ for $n\ge 2$, thus  the right hand side of \re{4.96} is less than $Q(n) \rho_n^0(t)$ when $0\le r < R_*^0$. Here $Q(n)$ is a polynomial function of $n$.

Since $\rho_n^0(t)$ has different behaviors under $n\ge 2$, $n=0$ and $n=1$, we divide the following procedures into three cases: (i) $n \ge 2$; (ii) $n = 0$; and (iii) $n=1$.

\vspace{15pt}

{\bf Case 1: When $n\ge 2$}\\

\vspace{-10pt}
For $n\ge 2$, we introduce the following lemma to estimate $u_n^{(1)}$.

\begin{lem}\label{thm4.6}
Consider the elliptic problem
\begin{eqnarray}
L_n w = -\frac{\partial^2 w}{\partial r^2} - \frac{1}{r}\frac{\partial w}{\partial r} + \frac{n^2}{r^2} w &=& b(r,t) \hspace{2em} \text{in } B_R,\label{4.97}\\
w|_{r=R} &=& 0,\label{4.98}
\end{eqnarray}
where $n \ge 2$.
If $b(\cdot,t)\in L^2(B_R)$, then this problem admits a unique solution $w$ in $H^2(B_R)$ with estimates
\begin{equation}\label{4.99}
    \|w(\cdot, t)\|_{H^2(B_R)} \le C\Big [\int_0^R |b(r,t)|^2 r \dif r\Big ]^{1/2};
\end{equation}
\begin{equation}\label{4.100}
    \Big| \frac{\partial w(R,t)}{\partial r}\Big | \le C\Big [\int_0^R |b(r,t)|^2 r \dif r\Big ]^{1/2},
\end{equation}
where the constant $C$ in \re{4.99} and \re{4.100} is independent of $n$.
\end{lem}

\begin{proof}
Let us consider the approximate equation to \re{4.97} in $\epsilon < r <R$ with zero boundary values on $x=R$ and $x=\epsilon$, where $\epsilon > 0 $ is arbitrarily small. We denote by $w_\epsilon$ the corresponding classical solution. Multiplying \re{4.97} by $\frac{w_\epsilon}{r^2}$ and integrating over $B_R \backslash B_\epsilon$, we obtain:
\begin{equation*}
\begin{split}
    \int_\epsilon^R \Big|\frac{\partial w_\epsilon}{\partial r}\Big|^2 \frac{1}{r}\dif r + n^2\int_\epsilon^R |w_\epsilon|^2 \frac{1}{r^3}\dif r =&\; \int_\epsilon^R b(r,t) \frac{w_\epsilon}{r}\dif r + 2\int_\epsilon^R \frac{\partial w_\epsilon}{\partial r}\frac{w_\epsilon}{r^2}\dif r \\
    \le &\; \frac{1}{2 n^2} \int_\epsilon^R |b(r,t)|^2 r \dif r + \frac{n^2}{2} \int_\epsilon^R |w_\epsilon|^2 \frac{1}{r^3}\dif r\\
    &\; + \frac23 \int_\epsilon^R \Big|\frac{\partial w_\epsilon}{\partial r}\Big|^2 \frac{1}{r}\dif r + \frac32 \int_\epsilon^R |w_\epsilon|^2 \frac{1}{r^3} \dif r,
    \end{split}
\end{equation*}
from which it follows that
\begin{equation}\label{4.101}
\frac13 \int_\epsilon^R \Big|\frac{1}{r}\frac{\partial w_\epsilon}{\partial r}\Big|^2 r \dif r + \Big(\frac{n^2}{2}-\frac32\Big) \int_\epsilon^R  \Big| \frac{w_\epsilon}{r^2}\Big|^2 r \dif r \le \frac{1}{2n^2}\int_\epsilon^R |b(r,t)|^2 r \dif r.
\end{equation}
The equation \re{4.97}, together with the fact that $b(\cdot,t)\in L^2(B_R)$, implies $\frac{\partial^2 w_\epsilon}{\partial r^2}\in L^2(B_R\backslash B_\epsilon)$. Therefore $w_\epsilon\in H^2(B_R\backslash B_\epsilon)$, and
\begin{equation}\label{4.102}
    \|w_\epsilon(\cdot, t)\|_{H^2(B_R\backslash B_\epsilon)} \le C\Big [\int_0^R |b(r,t)|^2 r \dif r\Big ]^{1/2},
\end{equation}
where $C$ is independent of $n$. Letting $\epsilon \rightarrow 0$, we obtain a solution $w$ to \re{4.97} and \re{4.98} with estimate \re{4.99}. The uniqueness of solution $w$ in $H^2(B_R)$ follows by taking $b=0$ and using \re{4.99}.

Next, since $H^2(B_R\backslash B_{R/2})\hookrightarrow C^{1+1/2}(B_R\backslash B_{R/2})$, we have
$$\| w(\cdot,t)\|_{C^{1+1/2}(B_R\backslash B_{R/2})} \le C\Big [\int_0^R |b(r,t)|^2 r \dif r\Big ]^{1/2},$$
which immediately implies \re{4.100}.
\end{proof}

Applying Lemma \ref{thm4.6} on \re{4.96}, we obtain when $n\ge 2$,
$$\| u_n^{(1)}\|_{H^2(B_{R_*^0})}  \le Q(n)|\rho_n^0(t)|,$$
$$\Big| \frac{\partial u_n^{(1)}(R_*^0,t)}{\partial r}\Big| \le Q(n)|\rho_n^0(t)|.$$
From Lemma \ref{thm4.5}, we know that when $n\ge 2$  and $\mu < \mu_*$,  there exists a constant $\delta > 0$ such that $|\rho_n^0(t)| \le |\rho_n^0(0)|e^{-\delta n^3 t} \text{ for all } t>0.$ It follows that
\begin{equation}
    \label{4.103}\| u_n^{(1)}\|_{H^2(B_{R_*^0})}  \le Q(n)|\rho_n^0(t)| \le C e^{-\delta n^3 t},
\end{equation}
\begin{equation}
    \label{4.104}\Big| \frac{\partial u_n^{(1)}(R_*^0,t)}{\partial r}\Big|\le Q(n)|\rho_n^0(t)| \le C e^{-\delta n^3 t}.
\end{equation}
Similarly, we can derive the same estimates for $u_n^{(2)}$ and $u_n^{(3)}$, namely,
\begin{equation}
    \label{4.105}\| u_n^{(2)}\|_{H^2(B_{R_*^0})}   \le C e^{-\delta n^3 t},
\end{equation}
\begin{equation}
    \label{4.106}\| u_n^{(3)}\|_{H^2(B_{R_*^0})}   \le C e^{-\delta n^3 t},
\end{equation}
\begin{equation}
    \label{4.107}\Big| \frac{\partial u_n^{(2)}(R_*^0,t)}{\partial r}\Big| \le C e^{-\delta n^3 t},
\end{equation}
\begin{equation}
    \label{4.108}\Big| \frac{\partial u_n^{(3)}(R_*^0,t)}{\partial r}\Big| \le C e^{-\delta n^3 t}.
\end{equation}
In order to take summation with respect to $n$ in the perturbed Fourier series, we restrict \re{4.103} --- \re{4.108} to $t\ge t_0$ for some small positive constant $t_0$. With these estimates in hand, the stability for the case $\mu < \mu_*$ will be determined by $u_n^{(4)}$. The solution $u_n^{(4)}$ to \re{4.95} is clearly given in the form,
\begin{equation}
    \label{4.109} u_n^{(4)}(r,t) = C_4(t) r^n,
\end{equation}
where $C_4(t)$ is determined by the boundary condition in \re{4.95}. Combining the boundary condition from \re{4.84} and \re{4.89}, we have
\begin{gather*}
    C_4(t)(R_*^0)^n = q_n^1(R_*^0,t) + \mu w_n^1(R_*^0,t) = \Big[ \frac{n^2 -1}{(R_*^0)^2} - \mu\frac{I_1(R_*^0)}{I_0(R_*^0)}\Big] \rho_n^1(t) + H(R_*^0,R_*^1)\rho_n^0(t),\\
    C_4(t) = \frac{1}{(R_*^0)^n}\Big[ \frac{n^2 -1}{(R_*^0)^2} - \mu\frac{I_1(R_*^0)}{I_0(R_*^0)}\Big] \rho_n^1(t) + \tilde{H}(R_*^0, R_*^1)\rho_n^0(t),
\end{gather*}
where $H$ and $\tilde{H}$ are functions of $R_*^0$ and $R_*^1$. Now let us combine $u_n^{(1)}, u_n^{(2)}, u_n^{(3)}$ and $u_n^{(4)}$ together,
\begin{equation*}
    q_n^1(r,t) =  \eta_n^1 - \mu w_n^1 = u_n^{(1)} + u_n^{(2)} + u_n^{(3)} + u_n^{(4)} - \mu w_n^1.
\end{equation*}
Note that we only need to evaluate $\frac{\partial q_n^1}{\partial r}(R_*^0,t)$, thus
\begin{equation}\label{4.110}
    \frac{\partial q_n^1(R_*^0,t)}{\partial r} = \frac{\partial u_n^{(1)}(R_*^0,t)}{\partial r} + \frac{\partial u_n^{(2)}(R_*^0,t)}{\partial r} + \frac{\partial u_n^{(3)}(R_*^0,t)}{\partial r} + \frac{\partial u_n^{(4)}(R_*^0,t)}{\partial r} - \mu \frac{\partial w_n^1(R_*^0,t)}{\partial r}.
\end{equation}

So far, we have calculated all the expressions needed in \re{4.85}. To get the equation for $\rho_n^1(t)$, we substitute $\frac{\partial^2 p_*^0}{\partial r^2}$ from \re{4.60}, $\frac{\partial^3 p_*^0}{\partial r^3}$ from \re{4.61}, $\frac{\partial^2 p_*^1}{\partial r^2}$ from \re{4.87}, $\frac{\partial^2 q_n^0}{\partial r^2}$ from \re{4.65}, and $\frac{\partial q_n^1}{\partial r}$ from \re{4.110}, into \re{4.85}, and get for $n\ge2$,
\begin{equation*}
\begin{split}
    \frac{\dif \rho_n^1(t)}{\dif t}
    =&\; -\frac{\partial^2 p_*^0}{\partial r^2}(R_*^0)\rho_n^1(t)-\frac{\partial^3 p_*^0}{\partial r^3}(R_*^0)R_*^1 \rho_n^0(t)-\frac{\partial^2 p_*^1}{\partial r^2}(R_*^0)\rho_n^0(t)-\frac{\partial^2 q_n^0}{\partial r^2}(R_*^0,t)R_*^1 - \frac{\partial q_n^1}{\partial r}(R_*^0,t)\\
    =&\; \Big[  \mu - \mu\frac{2I_1(R_*^0)}{R_*^0 I_0(R_*^0)} - \frac{n(n^2-1)}{(R_*^0)^3} - \mu\frac{I_1(R_*^0)I_{n+1}(R_*^0)}{I_0(R_*^0)I_n(R_*^0)}\Big] \rho_n^1(t) + C(n,R_*^0,R_*^1)\rho_n^0(t)\\
    &\; - \frac{\partial u_n^{(1)}(R_*^0,t)}{\partial r} - \frac{\partial u_n^{(2)}(R_*^0,t)}{\partial r} - \frac{\partial u_n^{(3)}(R_*^0,t)}{\partial r},\\
    \end{split}
\end{equation*}
thus,
\begin{equation}\label{4.111}
    \begin{split}
        &\; \Big|\frac{\dif \rho_n^1(t)}{\dif t} + \Big( \mu\frac{2I_1(R_*^0)}{R_*^0 I_0(R_*^0)} + \frac{n(n^2-1)}{(R_*^0)^3} + \mu\frac{I_1(R_*^0)I_{n+1}(R_*^0)}{I_0(R_*^0)I_n(R_*^0)}-\mu\Big) \rho_n^1(t)\Big|\\
        &\;\le |C(n,R_*^0,R_*^1)\rho_n^0(t)| + \Big|\frac{\partial u_n^{(1)}(R_*^0,t)}{\partial r}\Big| + \Big|\frac{\partial u_n^{(2)}(R_*^0,t)}{\partial r}\Big| + \Big| \frac{\partial u_n^{(3)}(R_*^0,t)}{\partial r}\Big|\\
        &\; \le C e^{-\delta n^3 t}.
    \end{split}
\end{equation}
To further analyze \re{4.111}, we introduce the following lemma.

\begin{lem}\label{thm4.7}
Suppose $f(t)$ satisfies
\begin{equation}
    \label{4.112}
    \Big| \frac{\dif f(t)}{\dif t} + d_1 f(t)\Big| \le Ce^{-d_2t}, \hspace{2em}\forall t>0
\end{equation}
with $d_1 \neq d_2 >0$ and the initial value $|f(0)|$ bounded, then we have
\begin{equation}
    \label{4.113}
    |f(t)| \le Ce^{-dt},
\end{equation}
where $d = \min\{d_1,d_2\}$.
\end{lem}

\begin{proof}
\re{4.112} is equivalent to
$$-C e^{-d_2 s} \le \frac{\dif f(s)}{\dif s} + d_1 f(s) \le C e^{-d_2 s},$$
and thus
$$-C e^{(d_1-d_2)s} \le \frac{\dif (e^{d_1s}f(s))}{\dif s} \le C e^{(d_1-d_2)s}.$$
Integrating $s$ from 0 to $t$, we derive for any $t>0$,
$$-C \int_0^t e^{(d_1-d_2)s}\dif s \le e^{d_1 t}f(t) - f(0)\le C\int_0^t e^{(d_1-d_2)s} \dif s,$$
$$\Big[f(0)+\frac{C}{d_1-d_2}\Big] e^{-d_1 t} - \frac{C}{d_1-d_2}e^{-d_2 t} \le f(t) \le \Big[f(0)-\frac{C}{d_1-d_2}\Big] e^{-d_1 t} + \frac{C}{d_1-d_2}e^{-d_2 t}.$$
If $d_1 >d_2 >0$, $e^{-d_1 t} < e^{-d_2 t}$, the above equation implies $|f(t)| \le C e^{-d_2 t}$; if $d_1 < d_2$, the above equation implies $|f(t)| \le C e^{-d_1 t}$.
\end{proof}

From Lemma \ref{thm4.5}, we have for $\mu < \mu_*$,
$$\mu\frac{2I_1(R_*^0)}{R_*^0 I_0(R_*^0)} + \frac{n(n^2-1)}{(R_*^0)^3} + \mu\frac{I_1(R_*^0)I_{n+1}(R_*^0)}{I_0(R_*^0)I_n(R_*^0)}-\mu > \delta n^3 > 0.$$
Therefore we apply Lemma \ref{thm4.7} on \re{4.111}, and get $|\rho_n^1(t)| \le Ce^{-\delta n^3 t}$, i.e., $|\rho_n^1(t)|$ is exponentially decreasing.

\vspace{15pt}
{\bf Case 2: When $n = 0$}\\
For $n = 0$, the estimates (4.71) and (4.72) follow from the standard $L^2$ and Schauder theory for elliptic equations. By Lemma 4.1,  there exists $\delta>0$ such that $$|\rho_0^0(t)| \le C e^{-\delta t}.$$
Following similar procedures as in case 1, we can derive
\begin{equation}
\label{4.114}
    \Big|\frac{\dif \rho_0^1(t)}{\dif t} + \mu\Big(-1 + \frac{2 I_1(R_*^0)}{R_*^0 I_0(R_*^0)} +  \frac{I_1^2(R_*^0)}{I_0^2(R_*^0)}\Big) \rho_0^1(t)\Big| \le C e^{-\delta t}.
\end{equation}
Again, from Lemma \ref{thm4.1},
$$\mu\Big(-1 + \frac{2 I_1(R_*^0)}{R_*^0 I_0(R_*^0)} +  \frac{I_1^2(R_*^0)}{I_0^2(R_*^0)}\Big) > \delta >0.$$
Using Lemma \ref{thm4.7}, we have $|\rho_0^1(t)| \le Ce^{-\delta t}$ , in other words, $|\rho_0^1(t)|$ is also exponentially decreasing.

\vspace{10pt}
{\bf Case 3: When $n = 1$}

\begin{thm}
For $n=1$ and any $\mu > 0$, we have $\rho_1^1(t) = \rho_1^1(0)$, for all $t>0$.
\end{thm}
\begin{proof}
For $n=1$, from \re{4.58}, \re{4.62}, and \re{4.63}, we have
\begin{gather}
w_1^0(r,t) = -\frac{I_1(r)}{I_0(R_*^0)}\rho_1^0(t),\label{4.115}\\
q_1^0(r,t) = -\mu\frac{I_1(R_*^0)}{R_*^0 I_0(R_*^0)}\rho_1^0(t)r + \mu\frac{I_1(r)}{I_0(R_*^0)}\rho_1^0(t).\label{4.116}
\end{gather}
Differentiating $q_1^0(r,t)$ with respect to $r$ twice and evaluating at $r=R_*^0$, using also \re{4.22}, we obtain
\begin{equation}
    \label{4.117}
    \frac{\partial^2 q_1^0}{\partial r^2}(R_*^0,t) = \mu\rho_1^0(t)\Big[-\frac{1}{R_*^0} + \frac{2I_1(R_*^0)}{(R_*^0)^2 I_0(R_*^0)} + \frac{I_1(R_*^0)}{I_0(R_*^0)}\Big].
\end{equation}
We already derived the formula of $w_1^1(r,t)$ in \re{4.89}. Using \re{4.25} to simplify, we have
\begin{equation}
    \label{4.118} w_1^1(r,t) = \Big[ \frac{I_1(R_*^0)}{I_0^2(R_*^0)}R_*^1 \rho_1^0(t) - \frac{1}{I_0(R_*^0)}\rho_1^1(t)\Big] I_1(r).
\end{equation}
Next let us find the expression for $q_1^1(r,t)$.
Substituting \re{4.115}, \re{4.116}, and \re{4.118} into \re{4.84}, noting that $\frac{\partial w_1^0}{\partial t}= -\frac{I_1(r)}{I_0(R_*^0)}\frac{\partial \rho_1^0(t)}{\partial t} = 0$ by Lemma \ref{thm4.2}, and using also \re{4.22}, \re{4.23}, \re{4.25}, \re{4.51}, \re{4.52}, and \re{4.57}, we derive the equation for $q_1^1(r,t)$
\begin{equation}
\begin{split}
    \label{4.119}
    -\frac{\partial^2 q_1^1}{\partial r^2}-\frac1r\frac{\partial q_1^1}{\partial r}+\frac{1}{r^2}q_1^1 =& -\mu^2\rho_1^0(t)\frac{I_1(R_*^0)}{R_*^0 I_0^2(R_*^0)} rI_0(r)+ \mu^2\rho_1^0(t)\frac{2}{I_0^2(R_*^0)}I_1(r)\Big[ I_0(r)-\frac{I_1(r)}{r}\Big]\\ &+ \mu\Big[ \frac{I_1(R_*^0)}{I_0^2(R_*^0)}R_*^1 \rho_1^0(t) - \frac{1}{I_0(R_*^0)}\rho_1^1(t)\Big] I_1(r),
    \end{split}
\end{equation}
with boundary condition
\begin{equation}
    \label{4.120}
    q_1^1(R_*^0,t) = -\mu\frac{I_2(R_*^0)}{I_0(R_*^0)}R_*^1\rho_1^0(t).
\end{equation}
To solve ODE \re{4.119} and \re{4.120}, we separate the solution into particular solutions $u_1^{(1)}$, $u_1^{(2)}$, $u_1^{(3)}$ and general solution $u_1^{(4)}$, where $u_1^{(1)}$, $\cdots$, $u_1^{(4)}$ satisfy the following equations, respectively,
\begin{gather}
    -\frac{\partial^2 u_1^{(1)}}{\partial r^2}-\frac1r\frac{\partial u_1^{(1)}}{\partial r}+\frac{1}{r^2}u_1^{(1)} = -\mu^2\rho_1^0(t)\frac{I_1(R_*^0)}{R_*^0 I_0^2(R_*^0)} rI_0(r),\label{4.121}\\
    -\frac{\partial^2 u_1^{(2)}}{\partial r^2}-\frac1r\frac{\partial u_1^{(2)}}{\partial r}+\frac{1}{r^2}u_1^{(2)} = \mu^2\rho_1^0(t)\frac{2}{I_0^2(R_*^0)}I_1(r)\Big[ I_0(r)-\frac{I_1(r)}{r}\Big],\label{4.122}\\
    -\frac{\partial^2 u_1^{(3)}}{\partial r^2}-\frac1r\frac{\partial u_1^{(3)}}{\partial r}+\frac{1}{r^2}u_1^{(3)} = \mu\Big[ \frac{I_1(R_*^0)}{I_0^2(R_*^0)}R_*^1 \rho_1^0(t) - \frac{1}{I_0(R_*^0)}\rho_1^1(t)\Big] I_1(r),\label{4.123}\\
    -\frac{\partial^2 u_1^{(4)}}{\partial r^2}-\frac1r\frac{\partial u_1^{(4)}}{\partial r}+\frac{1}{r^2}u_1^{(4)} = 0, \hspace{1em} u_1^{(1)} + u_1^{(2)} + u_1^{(3)} + u_1^{(4)}\Big|_{r=R_*^0}=-\mu\frac{I_2(R_*^0)}{I_0(R_*^0)}R_*^1\rho_1^0(t).\label{4.124}
\end{gather}
From the properties of Bessel function \re{4.22}--\re{4.25}, the functions $y_1(r) = \frac{r}{2}(1-2I_2(r))$, $y_2(r) = \frac{r}{4}(-I_1^2(r)+I_0(r)I_2(r))$ and $y_3(r) = -I_1(r)$ satisfy the following equations
$$-y_1'' - \frac1r y_1' + \frac{1}{r^2}y_1= rI_0(r), \hspace{1em} y_1' = \frac{1}{2}+I_2(r)-rI_1(r),$$
$$-y_2'' - \frac1r y_2' + \frac{1}{r^2}y_2 = I_1(r)\Big[ I_0(r)-\frac{I_1(r)}{r}\Big],\hspace{1em} y_2'=-\frac{I_0^2(r)}{4}+\frac{I_0(r)I_1(r)}{2r} - \frac{I_1^2(r)}{4},$$
$$-y_3'' - \frac1r y_3' + \frac{1}{r^2}y_3 = I_1(r),\hspace{1em} y_3'=-I_0(r)+\frac{I_1(r)}{r}.$$
Thus, we obtain
\begin{gather}
    u_1^{(1)} = -\mu^2\rho_1^0(t)\frac{I_1(R_*^0)}{R_*^0 I_0^2(R_*^0)}y_1(r) = -\mu^2\rho_1^0(t)\frac{I_1(R_*^0)}{R_*^0 I_0^2(R_*^0)}\frac{r}{2}(1-2I_2(r)),\label{4.125}\\
    u_2^{(2)} = \mu^2\rho_1^0(t)\frac{2}{I_0^2(R_*^0)}y_2(r) = \mu^2\rho_1^0(t)\frac{1}{2I_0^2(R_*^0)}r(-I_1^2(r)+I_0(r)I_2(r)),\label{4.126}\\
    u_3^{(3)} = \mu\Big[ \frac{I_1(R_*^0)}{I_0^2(R_*^0)}R_*^1 \rho_1^0(t) - \frac{1}{I_0(R_*^0)}\rho_1^1(t)\Big]y_3(r) =\mu\Big[ \frac{1}{I_0(R_*^0)}\rho_1^1(t)-\frac{I_1(R_*^0)}{I_0^2(R_*^0)}R_*^1 \rho_1^0(t)\Big]I_1(r).\label{4.127}
\end{gather}
In addition, the general solution $u_1^{(4)}$ takes the form of
\begin{equation}
    \label{4.128}
    u_1^{(4)} = r.
\end{equation}
Combining \re{4.125}, \re{4.126}, \re{4.127}, and \re{4.128}, we find that the solution to \re{4.119} should be
\begin{equation}
    \label{4.129}
    q_1^1(r,t) = C_5(t)u_1^{(4)} + u_1^{(1)} + u_1^{(2)} + u_1^{(3)},
\end{equation}
where $C_5(t)$ is determined by the boundary condition in \re{4.124}, namely,
\begin{equation*}
    \begin{split}
        C_5(t) R_*^0 =&\; q_1^1(R_*^0,t) - u_1^{(1)}(R_*^0) - u_1^{(2)}(R_*^0) - u_1^{(3)}(R_*^0)\\
        =&\; -\mu R_*^1\rho_1^0(t)\frac{I_2(R_*^0)}{I_0(R_*^0)} + \mu^2\rho_1^0(t)\frac{I_1(R_*^0)}{ 2I_0^2(R_*^0)}\Big[1-2I_2(R_*^0)\Big] - \mu^2\rho_1^0(t)\frac{R_*^0}{2I_0^2(R_*^0)}\Big[-I_1^2(R_*^0)\\
        &\; +I_0(R_*^0)I_2(R_*^0)\Big]
        - \mu\Big[ \frac{1}{I_0(R_*^0)}\rho_1^1(t)-\frac{I_1(R_*^0)}{I_0^2(R_*^0)}R_*^1 \rho_1^0(t)\Big]I_1(R_*^0),
    \end{split}
\end{equation*}
which simplifies to
\begin{equation}
    \label{4.130}
    \begin{split}
        C_5(t) =& \;-\mu\rho_1^1(t)\frac{I_1(R_*^0)}{R_*^0 I_0(R_*^0)} +\mu^2\rho_1^0(t) \frac{I_1(R_*^0)}{2R_*^0 I_0^2(R_*^0)} -\mu^2\rho_1^0(t) \frac{I_1(R_*^0)I_2(R_*^0)}{R_*^0 I_0^2(R_*^0)}\\ &\; + \mu^2\rho_1^0(t)\frac{ I_1^2(R_*^0)}{2 I_0^2(R_*^0)}
         -\mu^2\rho_1^0(t)\frac{ I_2(R_*^0)}{2 I_0(R_*^0)}+ \mu R_*^1\rho_1^0(t)\frac{I_1^2(R_*^0)}{R_*^0 I_0^2(R_*^0)} - \mu R_*^1\rho_1^0(t)\frac{ I_2(R_*^0)}{R_*^0 I_0(R_*^0)}.
    \end{split}
\end{equation}

In order to compute $\rho_1^1(t)$ from \re{4.85}, we need the derivative value $\frac{\partial q_1^1}{\partial r}(R_*^0,t)$. We combine \re{4.128} and \re{4.129} to obtain
\begin{equation*}
    \begin{split}
        \frac{\partial q_1^1}{\partial r}(R_*^0,t) =& \;C_5(t) \frac{\partial u_1^{(4)}}{\partial r}(R_*^0) + \frac{\partial u_1^{(1)}}{\partial r}(R_*^0) + \frac{\partial u_1^{(2)}}{\partial r}(R_*^0) + \frac{\partial u_1^{(3)}}{\partial r}(R_*^0)\\
        =&\;C_5(t) + \frac{\partial u_1^{(1)}}{\partial r}(R_*^0) + \frac{\partial u_1^{(2)}}{\partial r}(R_*^0) + \frac{\partial u_1^{(3)}}{\partial r}(R_*^0).
    \end{split}
\end{equation*}
Applying \re{4.125}, \re{4.126}, \re{4.127}, and \re{4.130}, we then derive
\begin{equation}
\begin{aligned}
    \label{4.131}
    \frac{\partial q_1^1}{\partial r}(R_*^0,t) =&  -\mu\rho_1^1(t) \frac{I_1(R_*^0)}{R_*^0 I_0(R_*^0)} +\mu^2 \rho_1^0(t) \frac{I_1(R_*^0)}{2R_*^0 I_0^2(R_*^0)} -\mu^2\rho_1^0(t) \frac{I_1(R_*^0)I_2(R_*^0)}{R_*^0 I_0^2(R_*^0)} + \mu^2 \rho_1^0(t)\frac{ I_1^2(R_*^0)}{2 I_0^2(R_*^0)}\\
        & -\mu^2\rho_1^0(t)\frac{ I_2(R_*^0)}{2 I_0(R_*^0)}+ \mu R_*^1\rho_1^0(t) \frac{ I_1^2(R_*^0)}{R_*^0 I_0^2(R_*^0)} - \mu R_*^1\rho_1^0(t) \frac{ I_2(R_*^0)}{R_*^0 I_0(R_*^0)}
        -\mu^2\rho_1^0(t)\frac{I_1(R_*^0)}{2 R_*^0 I_0^2(R_*^0)} \\
        & - \mu^2\rho_1^0(t)\frac{I_1(R_*^0)I_2(R_*^0)}{R_*^0 I_0^2(R_*^0)} + \mu^2\rho_1^0(t) \frac{I_1^2(R_*^0)}{I_0^2(R_*^0)} - \frac{\mu^2}{2}\rho_1^0(t) + \mu^2\rho_1^0(t)\frac{I_1(R_*^0)}{R_*^0 I_0(R_*^0)}\\ &-\mu^2\rho_1^0(t)\frac{I_1^2(R_*^0)}{2 I_0^2(R_*^0)}
         +\mu\rho_1^1(t) - \mu\rho_1^1(t)\frac{I_1(R_*^0)}{R_*^0 I_0(R_*^0)} - \mu R_*^1\rho_1^0(t)\frac{ I_1(R_*^0)}{I_0(R_*^0)} + \mu R_*^1\rho_1^0(t)\frac{ I_1^2(R_*^0)}{R_*^0 I_0^2(R_*^0)}\\
        =&\; \mu\rho_1^1(t)\Big[-\frac{2I_1(R_*^0)}{R_*^0 I_0(R_*^0)} + 1\Big] + \mu R_*^1 \rho_1^0(t)\Big[\frac{2 I_1^2(R_*^0)}{R_*^0 I_0^2(R_*^0)} - \frac{I_2(R_*^0)}{R_*^0 I_0(R_*^0)} - \frac{I_1(R_*^0)}{I_0(R_*^0)}\Big] \\
        &+ \mu^2 \rho_1^0(t)\Big[-\frac{2I_1(R_*^0)I_2(R_*^0)}{R_*^0 I_0^2(R_*^0)} - \frac{I_2(R_*^0)}{2I_0(R_*^0)} + \frac{I_1^2(R_*^0)}{I_0^2(R_*^0)} - \frac12 + \frac{I_1(R_*^0)}{R_*^0 I_0(R_*^0)}\Big].
        \end{aligned}
\end{equation}

Finally, it remains to substitute $\frac{\partial^2 p_*^0}{\partial r^2}$ from \re{4.60}, $\frac{\partial^3 p_*^0}{\partial r^3}$ from \re{4.61}, $\frac{\partial^2 p_*^1}{\partial r^2}$ from \re{4.87}, $\frac{\partial^2 q_1^0}{\partial r^2}$ from \re{4.117}, and $\frac{\partial q_1^1}{\partial r}$ from \re{4.131}, into equation \re{4.85}, and collect terms containing $\rho_1^0(t)$ and $\rho_1^1(t)$,
\begin{equation}\label{4.132}
    \begin{split}
        \frac{\dif \rho_1^1(t)}{\dif t}=&\; -\frac{\partial^2 p_*^0}{\partial r^2}(R_*^0)\rho_1^1(t)-\frac{\partial^3 p_*^0}{\partial r^3}(R_*^0)R_*^1 \rho_1^0(t)-\frac{\partial^2 p_*^1}{\partial r^2}(R_*^0)\rho_1^0(t)-\frac{\partial^2 q_1^0}{\partial r^2}(R_*^0,t)R_*^1 - \frac{\partial q_1^1}{\partial r}(R_*^0,t) \\
    =&\; -\mu\rho_1^1(t)\Big[\frac{2I_1(R_*^0)}{R_*^0 I_0(R_*^0)}-1\Big] - \mu R_*^1 \rho_1^0(t)\Big[ \frac{1}{R_*^0} - \frac{2I_1(R_*^0)}{(R_*^0)^2 I_0(R_*^0)} - \frac{I_1(R_*^0)}{I_0(R_*^0)}\Big] - \mu^2\rho_1^0(t)\frac{I_1(R_*^0)}{R_*^0 I_0^2(R_*^0)}\Big[\\
    &\; -R_*^0 I_1(R_*^0)+ I_2(R_*^0)\Big] - \mu^2\rho_1^0(t) \frac{1}{2I_0^2(R_*^0)}\Big[I_0^2(R_*^0)-\frac{2I_0(R_*^0)I_1(R_*^0)}{R_*^0} + I_1^2(R_*^0)\Big]\\
    &\; -\mu R_*^1\rho_1^0(t)\frac{ I_1(R_*^0)}{I_0^2(R_*^0)}\Big[I_0(R_*^0) - \frac{I_1(R_*^0)}{R_*^0}\Big] - \mu R_*^1\rho_1^0(t)\Big[-\frac{1}{R_*^0} + \frac{2I_1(R_*^0)}{(R_*^0)^2 I_0(R_*^0)} + \frac{I_1(R_*^0)}{I_0(R_*^0)}\Big]\\
    &\; - \mu\rho_1^1(t)\Big[-\frac{2I_1(R_*^0)}{R_*^0 I_0(R_*^0)} + 1\Big] - \mu R_*^1 \rho_1^0(t)\Big[\frac{2 I_1^2(R_*^0)}{R_*^0 I_0^2(R_*^0)} - \frac{I_2(R_*^0)}{R_*^0 I_0(R_*^0)} - \frac{I_1(R_*^0)}{I_0(R_*^0)}\Big] \\
        &\; - \mu^2 \rho_1^0(t)\Big[-\frac{2I_1(R_*^0)I_2(R_*^0)}{R_*^0 I_0^2(R_*^0)} - \frac{I_2(R_*^0)}{2I_0(R_*^0)} + \frac{I_1^2(R_*^0)}{I_0^2(R_*^0)} - \frac12 + \frac{I_1(R_*^0)}{R_*^0 I_0(R_*^0)}\Big]\\
    =& \;\mu\rho_1^1(t)\Big[1-\frac{2 I_1(R_*^0)}{R_*^0 I_0(R_*^0)} + \frac{2 I_1(R_*^0)}{R_*^0 I_0(R_*^0)} -1\Big]
    + \mu R_*^1 \rho_1^0(t)\Big[-\frac{1}{R_*^0} + \frac{2 I_1(R_*^0)}{(R_*^0)^2 I_0(R_*^0)} + \frac{I_1(R_*^0)}{I_0(R_*^0)} \\
    &\; - \frac{I_1(R_*^0)}{I_0(R_*^0)} + \frac{I_1^2(R_*^0)}{R_*^0 I_0^2(R_*^0)} + \frac{1}{R_*^0} - \frac{2 I_1(R_*^0)}{(R_*^0)^2 I_0(R_*^0)} - \frac{I_1(R_*^0)}{I_0(R_*^0)} - \frac{2I_1^2(R_*^0)}{R_*^0 I_0^2(R_*^0)} + \frac{I_2(R_*^0)}{R_*^0 I_0(R_*^0)}+\frac{I_1(R_*^0)}{I_0(R_*^0)}\Big]\\
    &\; +\mu^2 \rho_1^0(t)\Big[\frac{I_1^2(R_*^0)}{I_0^2(R_*^0)} - \frac{I_1(R_*^0)I_2(R_*^0)}{R_*^0 I_0^2(R_*^0)} -\frac{1}{2} + \frac{I_1(R_*^0)}{R_*^0 I_0(R_*^0)} -\frac{I_1^2(R_*^0)}{2I_0^2(R_*^0)} +\frac{2I_1(R_*^0)I_2(R_*^0)}{R_*^0 I_0^2(R_*^0)} \\
    &\; + \frac{I_2(R_*^0)}{2I_0(R_*^0)} - \frac{I_1^2(R_*^0)}{I_0^2(R_*^0)} + \frac12 - \frac{I_1(R_*^0)}{R_*^0 I_0(R_*^0)}\Big]\\
    =&\;  \mu R_*^1 \rho_1^0(t)\Big[-\frac{I_1^2(R_*^0)}{R_*^0 I_0^2(R_*^0)} + \frac{I_2(R_*^0)}{R_*^0 I_0(R_*^0)}\Big] + \mu^2\rho_1^0(t)\Big[ \frac{I_1(R_*^0)I_2(R_*^0)}{R_*^0 I_0^2(R_*^0)} - \frac{I_1^2(R_*^0)}{2I_0^2(R_*^0)} + \frac{I_2(R_*^0)}{2I_0(R_*^0)}\Big].
    \end{split}
\end{equation}
Recalling we have formula for $R_*^1$ in \re{4.76},
\begin{equation}\label{4.133}
    R_*^1 = \mu\frac{-2 I_1(R_*^0)I_2(R_*^0) + R_*^0 I_1^2(R_*^0) - R_*^0 I_0(R_*^0)I_2(R_*^0)}{2(-I_1^2(R_*^0) + I_0(R_*^0)I_2(R_*^0))},
\end{equation}
substituting it into \re{4.132}, we deduce
\begin{equation}
    \begin{split}\label{4.134}
         \frac{\dif \rho_1^1(t)}{\dif t}=\;&\frac{\mu R_*^1\rho_1^0(t)}{2R_*^0 I_0^2(R_*^0)}\Big[2(-I_1^2(R_*^0)
         + I_0(R_*^0)I_2(R_*^0))\Big] \\
         &\; + \frac{\mu^2\rho_1^0(t)}{2R_*^0 I_0^2(R_*^0)}\Big[2I_1(R_*^0)I_2(R_*^0)-R_*^0 I_1^2(R_*^0) + R_*^0 I_0(R_*^0)I_2(R_*^0)\Big]\\
         =&\; \frac{\mu^2 \rho_1^0(t)}{2 R_*^0 I_0^2(R_*^0)}\Big[-2 I_1(R_*^0)I_2(R_*^0) + R_*^0 I_1^2(R_*^0) - R_*^0 I_0(R_*^0)I_2(R_*^0) \\
         &+ 2I_1(R_*^0)I_2(R_*^0)-R_*^0 I_1^2(R_*^0) + R_*^0 I_0(R_*^0)I_2(R_*^0)\Big]\\
         =&\; 0.
    \end{split}
\end{equation}
Hence $\rho_1^1(t) = \rho_1^1(0)$, which completes the proof.
\end{proof}

\begin{rem}
We just established, for $n=1$, after ignoring $O(\tau^2)$ terms,
$$r=R_* + \epsilon \rho_1(t)\cos\theta = R_* + \epsilon(\rho_1^0(t)+\tau\rho_1^1(t))\cos\theta
= R_* + \epsilon(\rho_1^0(0)+\tau\rho_1^1(0))\cos\theta,$$
is a transform of the origin. Thus $n=1$ mode would not affect the stability.
\end{rem}

\bigskip

{\bf Acknowledgment.} We are grateful to the referees for a very careful reading of
our manuscript, and for helpful suggestions.

\end{document}